\documentclass[11pt,a4paper]{article}

\usepackage{amsmath}
\usepackage{amssymb,latexsym}
\usepackage{amsthm}
\usepackage[shortlabels]{enumitem}
\usepackage{array}
\usepackage[hang,flushmargin,bottom]{footmisc}
\usepackage[margin=0.35in,format=hang,font=small,labelfont=bf]{caption}
\usepackage[normalem]{ulem}
\usepackage{needspace}
\usepackage{graphicx,tikz}
\usepackage[colorlinks=true,urlcolor=blue,linkcolor=blue,citecolor=blue]{hyperref}
\usepackage{palatino}
\usepackage{eulervm}

\newtheorem{thm}{Theorem}[section]
\newtheorem*{thm*}{Theorem}
\newtheorem{prop}[thm]{Proposition}
\newtheorem{cor}[thm]{Corollary}
\newtheorem{lemma}[thm]{Lemma}

\newtheorem{obs}[thm]{Observation}
\newtheorem{defn}[thm]{Definition}
\newtheorem*{defn*}{Definition}
\newtheorem{conj}[thm]{Conjecture}
\newtheorem*{conj*}{Conjecture}

\theoremstyle{definition}

\newcommand{\AAA}{\mathcal{A}}
\newcommand{\BBB}{\mathcal{B}}
\newcommand{\CCC}{\mathcal{C}}

\newcommand{\DDD}{\mathcal{D}}

\newcommand{\FFF}{\mathcal{F}}

\newcommand{\HHH}{\mathcal{H}}

\newcommand{\LLL}{\mathcal{L}}

\newcommand{\PPP}{\mathcal{P}}

\newcommand{\RRR}{\mathcal{R}}
\newcommand{\SSS}{\mathcal{S}}
\newcommand{\TTT}{\mathcal{T}}

\newcommand{\Free}{\operatorname{Free}}

\newcommand{\cw}{\mathop{cw}}

\newcommand{\tw}{\mathop{tw}}

\newtheorem{example}[thm]{Example}
\newtheorem*{example*}{Example}

\newtheorem*{comment*}{Comment}

\usepackage{centernot}

\usetikzlibrary{matrix,shapes, arrows, positioning, automata,shapes.geometric,calc,backgrounds,er}
\usetikzlibrary{decorations.markings}

\tikzstyle{vertex}=[circle, draw, fill=black,
                        inner sep=0pt, minimum width=4pt]
\tikzstyle{every node}=[circle, draw, fill=black,
                        inner sep=0pt, minimum width=4pt]

\newcommand{\mnnodearray}[2]{ 
\foreach \i in {1,...,#1}
	\foreach \j in {1,...,#2}
		\node at (\i,\j) {};
}

\title{\texorpdfstring{$t$}{t{}}-sails and sparse hereditary classes of unbounded tree-width}

\author{D. Cocks\footnote{Supported by the Engineering and Physical Sciences Research Council [EP/V520147/1].}\\
\small School of Mathematics and Statistics\\
\small The Open University, UK}

\begin{document}
\maketitle

\begin{abstract}
It has long been known that the following basic objects are obstructions to bounded tree-width: for arbitrarily large $t$, $(1)$ the complete graph $K_t$, $(2)$ the complete bipartite graph $K_{t,t}$, $(3)$ a subdivision of the $(t \times t)$-wall and $(4)$ the line graph of a subdivision of the $(t \times t)$-wall. We  now add a further \emph{boundary object} to this list, a \emph{$t$-sail}. 

These results have been obtained by studying sparse hereditary  \emph{path-star} graph classes, each of which consists of the finite induced subgraphs of a single infinite graph whose edges can be partitioned into a path (or forest of paths) with a forest of stars, characterised by an infinite word over a possibly infinite alphabet. We show that a path-star class whose infinite graph has an unbounded number of stars, each of which  connects an unbounded number of times to the path, has unbounded tree-width. In addition, we show that such a class is not a subclass of the hereditary class of circle graphs.

We identify a collection of \emph{nested} words with a recursive structure that exhibit interesting characteristics when used to define a path-star graph class.  These  graph classes do not contain any of the four basic obstructions but instead contain graphs that have large tree-width if and only if they contain arbitrarily large $t$-sails. We show that these classes are infinitely defined and, like classes of bounded degree or classes excluding a fixed minor, do not contain a minimal class of unbounded tree-width.   
\end{abstract}

%
%
%
%
%
%
\section{Introduction} \label{Sect:Intro}

Tree-width is a graph parameter that became of great interest in the fields of structural and  algorithmic graph theory following the series of papers published by Robertson and Seymour on graph minors (for example \cite{robertson:graph-minors-ii:}). In particular, their Grid Minor Theorem \cite{robertson:graph-minors-v:} states that every graph of large enough tree-width  must contain a minor isomorphic to a large grid (or equivalently, a large wall). Consequently, tree-width has been primarily associated with minor-closed graph classes. However, recent interest in tree-width has focussed on hereditary graph classes, sparked by a paper by Aboulker, Adler, Kim, Sintiari and Trotignon containing the result:

\begin{thm}[Induced Grid Theorem for Minor-Free Graphs \cite{aboulker:tree-width-even-hole-free:}]\label{thm_minor_free} 
For every graph $H$ there is a function $f_H:\mathbb{N} \rightarrow \mathbb{N}$ such that every H-minor-free graph of tree-width at least $f_H(t)$ contains an induced subgraph isomorphic to a subdivision of a $(t \times t)$-wall or the line graph of a subdivision of a $(t \times t)$-wall.
\end{thm}

This was followed by a result from Korhonen:

\begin{thm} [Corollary to The Grid Induced Minor Theorem \cite{korhonen:bounded_degree:}]\label{thm_degree_bounded}
For every $\Delta \in \mathbb{N}$ there is a function $f_{\Delta}:\mathbb{N} \rightarrow \mathbb{N}$ such that every graph with degree at most $\Delta$ and tree-width at least $f_{\Delta}(t)$ contains an induced subgraph isomorphic to a subdivision of a $(t \times t)$-wall or the line graph of a subdivision of a $(t \times t)$-wall.
\end{thm}

Thus, a hereditary graph class with an excluded minor or of bounded vertex degree has unbounded tree-width if and only if it contains arbitrarily large subdivisions of a wall or the line graph of a subdivision of a wall.

It has long been known that the following \emph{$t$-basic obstructions} are obstructions to bounded tree-width: for arbitrarily large $t$, $(1)$ the complete graph $K_t$, $(2)$ the complete bipartite graph $K_{t,t}$, $(3)$ a subdivision of the $(t \times t)$-wall and $(4)$ the line graph of a subdivision of the $(t \times t)$-wall. It is tempting to conclude that these four objects are the only obstructions to bounded tree-width in sparse hereditary classes. However, counterexamples have recently been found by Sintiari and Trotignon \cite{sintiari:ttf-free:}, Davies \cite{davies:oberwolfach2022:} and Bonamy, Bonnet, D\'{e}pr\'{e}s, Esperet, Geniet, Hilaire, Thomass\'{e} and Wesolek \cite{Bonamy:sparse_graphs_log_treewidth:}. Tree-width in hereditary graph classes is currently a very active area of research. In particular, Abrishami, Alecu, Chudnovsky, Dibek, Hajebi, Rz\c{a}\.{z}ewski,  Spirkl and Vu\v{s}kovi\'{c} have contributed to a series of papers on the topic of induced subgraphs and tree decompositions, e.g. \cite{abrishami:induced_subgraphs_VIII:,abrishami:induced_subgraphs_VII:}.

A hereditary graph class is \emph{$KKW$-free} if there exists $t \in \mathbb{N}$ such that the class does not contain the complete graph $K_t$, the complete bipartite graph $K_{t,t}$,  a subdivision of the $(t\times t)$-wall or the line graph of a subdivision of the $(t\times t)$-wall. Our interest is in $KKW$-free classes that nevertheless have unbounded tree-width.

The least number of forests that can cover the edges of a graph is called its \emph{arboricity}. A graph with arboricity bounded by $k \in \mathbb{N}$ is called \emph{$k$-uniformly sparse} or just \emph{sparse}. A graph class is $k$-uniformly sparse if it does not contain a graph of arboricity greater than $k$. Uniformly sparse graph classes are a larger family than either  excluded minor or  bounded degree classes. 

Using Ramsey theory \cite[Proposition $9.4.1$]{diestel:graph-theory5:}  we know that for every $r \in \mathbb{N}$ there is an $n \in \mathbb{N}$ such that every connected graph of order at least $n$ contains a clique, a star or a path of order $r$ as an induced subgraph. Thus, large trees contain long induced paths or big induced stars, and this suggests that 'path-path' (two forests of paths), 'path-star'(a forest of paths and a forest of stars) and 'star-star' (two forests of stars) are the three natural structures to look for in graph classes of arboricity two. Given that walls are 'path-path' graphs we seek further obstructions to bounded tree-width in another family of classes of arboricity two, namely 'path-star' classes. 

\begin{defn}\label{def_path_star_graph}
A finite graph $(V,E)$ is a \emph{path-star} graph if its edges $E$ can be partitioned into two sets, $E_P$ and $E_S$, so that 
\begin{enumerate}
\item $(V,E_P)$ is a forest of paths, and
\item $(V,E_S)$ is a forest of stars.
\end{enumerate} 
\end{defn}

In this paper we consider a particular type of hereditary \emph{path-star class} being the collection of all the finite induced subgraphs of a single infinite graph whose edges can be partitioned into a path (or forest of paths) together with a forest of stars, where the leaves of the stars (but not the internal vertices of the stars) may embed in the paths. We denote a path-star class $\RRR^{\alpha}$ where $\alpha$ is an infinite word over alphabet $\mathbb{N}$ (see formal Definition \ref{def_path_star_class}). 
Our study of path-star graph classes has led to the discovery of a family of objects, \emph{$t$-sails}, with tree-width at least $t-1$ (Lemma \ref{lem_sail}).

\begin{defn}\label{def_tsail}
A path-star graph $(V,E_P \cup E_S)$ is a $t$-sail if there exists a set $V_S=\{s_1, \dots, s_t\}\subseteq V$ so that 
\begin{enumerate} 
\item No edge in $E_P$ is incident with a vertex in $V_S$,
\item The graph $(V \setminus  V_S,E_P)$ comprises $t$ components $P_1, \dots, P_t$ (all paths), and
\item  For all $1 \le i \le j \le t$ there exists $v \in P_i$ such that $s_jv \in E_S$.
\end{enumerate} 
\end{defn}
An example is shown in Figure \ref{fig:sail}.

\begin{figure}\centering

\begin{tikzpicture}[scale=1,
	vertex2/.style={circle,draw,minimum size=9,fill=blue},]

	\foreach \j in {1,...,7}
	\foreach \i in {1,...,\j}
		\node[red] at (\i,8-\j) {};
		
	\foreach \i in {2,...,8} \node[blue] at (0,\i) {};
	
	\foreach \j in {1,...,6}
	\foreach \i in {1,...,\j}	
		\draw (\i,7-\j) -- (\i+1,7-\j);
	
	\draw(0,8) edge[out=-30,in=120] (2,6);
	\draw(0,8) edge[out=-30,in=120] (3,5);
	\draw(0,8) edge[out=-30,in=120] (4,4);
	\draw(0,8) edge[out=-30,in=120] (5,3);
	\draw(0,8) edge[out=-30,in=120] (6,2);
	\draw(0,8) edge[out=-30,in=120] (7,1);
	\draw(0,7) edge[out=-30,in=120] (2,5);
	\draw(0,7) edge[out=-30,in=120] (3,4);	
	\draw(0,7) edge[out=-30,in=120] (4,3);
	\draw(0,7) edge[out=-30,in=120] (5,2);
	\draw(0,7) edge[out=-30,in=120] (6,1);
	\draw(0,6) edge[out=-30,in=120] (2,4);
	\draw(0,6) edge[out=-30,in=120] (3,3);
	\draw(0,6) edge[out=-30,in=120] (4,2);
	\draw(0,6) edge[out=-30,in=120] (5,1);
	\draw(0,5) edge[out=-30,in=120] (2,3);
	\draw(0,5) edge[out=-30,in=120] (3,2);
	\draw(0,5) edge[out=-30,in=120] (4,1);
	\draw(0,4) edge[out=-30,in=120] (2,2);
	\draw(0,4) edge[out=-30,in=120] (3,1);
	\draw(0,3) edge[out=-30,in=120] (2,1);
	\draw(0,8)--(1,7);\draw(0,7)--(1,6);\draw(0,6)--(1,5);
	\draw(0,5)--(1,4);\draw(0,4)--(1,3);\draw(0,3)--(1,2);
	\draw(0,2)--(1,1);

\end{tikzpicture}

\caption{A simple $7$-sail}
	\label{fig:sail}

\end{figure}
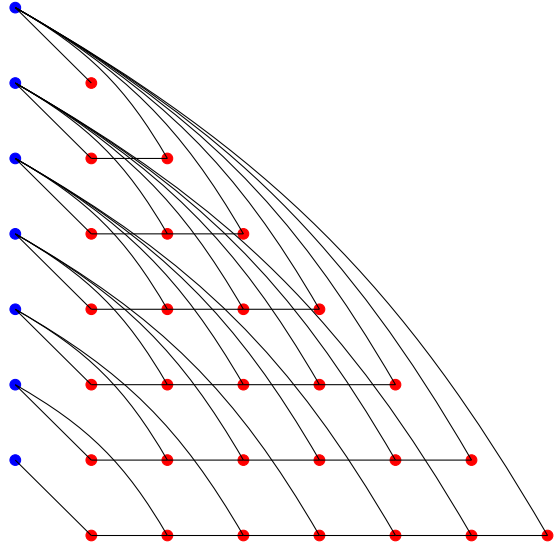

In fact, independently developed, $t$-sails are a generalisation of structures previously observed by Pohoata \cite{pohoata:unavoid:} and Davies \cite{davies:oberwolfach2022:}. In particular, in \cite{davies:oberwolfach2022:} a type of $t$-sail was used to disprove a conjecture that tree-width is always $O(log |V(G)|)$ for a graph $G$ in a hereditary class that is $KKW$-free.

We show that a path-star class, defined by an infinite word over an infinite alphabet where each letter in the alphabet appears an infinite number of times in the word,  contains a $t$-sail for arbitrarily large $t$ and therefore has unbounded tree-width (Theorem \ref{unbd_tw}). To distinguish this work from that recently undertaken on circle graphs by Hickingbotham, Illingworth, Mohar and Wood \cite{hickingbotham:treewidth_circlegraphs:}, we show that no path-star class defined by a word where at least two letters alternate more than four times  is a subclass of circle graphs (Theorem \ref{pathstar_circle}).

We identify a collection of \emph{nested} words (see Section \ref{words}) with a recursive structure that exhibit interesting characteristics when used to define a hereditary path-star  graph class: 

\newtheorem*{KKW_free}{Theorem \ref{KKW_free}}
\begin{KKW_free}
If $\alpha$ is a nested word then the path-star class $\RRR^{\alpha}$ is $KKW$-free.
\end{KKW_free}

\newtheorem*{thm_sail}{Theorem \ref{thm_sail}}
\begin{thm_sail}
If $\alpha$ is a nested word then for every $t \ge 1$ there is a positive integer valued function $f_{\alpha}(t)$ such that  every graph in $\RRR^{\alpha}$ of tree-width at least $f_{\alpha}(t)$ contains a $t$-sail as an induced subgraph. 
\end{thm_sail}

This shows that for a path-star class defined by a nested word, $t$-sails perform a similar role as walls do for minor excluded or bounded degree classes, in that they are the basic object obstructing bounded tree-width, albeit that in this case a $t$-sail is not one fixed graph like a wall but a more general construction.

The concept of \emph{well-quasi-ordering} (for definition see \cite{diestel:graph-theory5:} page 348) is central to the Graph Minor Theorem \cite{robertson:graph-minors-xx:}, that tells us that the finite graphs are well-quasi-ordered by the minor relation. Consequently, any minor-closed graph class, such as planar graphs, is finitely defined under the minor relationship -- that is, such a class can be defined by a finite number of minimal forbidden minors. Unfortunately, the same cannot be said of hereditary graph classes, which are not, in general, well-quasi-ordered.  The list of minimal forbidden induced subgraphs in a hereditary class may be finite or infinite.

Given the importance of finitely defined classes in the Graph Minor Theorem,  much recent research into obstructions to bounded tree-width in hereditary classes has been focussed on these. In particular, a paper from Lozin and Razgon characterizes hereditary classes of unbounded tree-width that are finitely defined \cite{lozin:tw_dichotomy:}. We show that the main result of \cite{lozin:tw_dichotomy:}, reproduced here as Theorem \ref{thm_dichotomy},  has the following consequence:

\newtheorem*{thm_infinite_def_1}{Theorem \ref{thm_infinite_def_1}}
\begin{thm_infinite_def_1}
A hereditary class of graphs of unbounded tree-width that is $KKW$-free  is infinitely defined.
\end{thm_infinite_def_1}

It has previously been shown (reproduced here as Theorem \ref{thm_tw_cw}) that in sparse graph classes tree-width and another parameter, clique-width, are either both bounded or unbounded, so the behaviour of clique-width may shed light on the behaviour of tree-width.

By contrast with sparse hereditary classes, all dense hereditary classes have unbounded tree-width (Theorem \ref{thm_kostochka}), but this is not true for clique-width (Section \ref{Clique_width}), where there are dense classes of both bounded and unbounded clique-width.

A hereditary class of graphs $\CCC$ is \emph{minimal of unbounded tree-width/clique-width} if every proper hereditary subclass $\DDD$ has bounded tree-width/clique-width (if it is clear from the context whether we are referring to tree-width or clique-width we will just call the class \emph{minimal}). In other words,  a hereditary graph class $\CCC$ is minimal if, for any proper hereditary subclass $\DDD$ formed by adding just one more forbidden graph, $\DDD$ has  bounded tree-width/clique-width.

The discovery of the first minimal hereditary classes of unbounded clique-width was made by Lozin~\cite{lozin:minimal-classes:}. However, more recently many more such classes have been identified by Atminas, Brignall, Lozin and Stacho~\cite{abls:minimal-classes-of:}, Collins, Foniok, Korpelainen, Lozin and Zamaraev~\cite{collins:infinitely-many:} and Dawar and Sankaran~\cite{dawar:clique_width:}. Most recently Brignall and Cocks demonstrated an uncountably infinite family of minimal hereditary classes of unbounded clique-width in \cite{brignall_cocks:uncountable:} and created a framework for minimal classes in \cite{brignall:framework_minimal_classes:}. 

 It is therefore natural to ask whether there are any sparse minimal classes of unbounded tree-width (or clique-width).

We show the following:

\newtheorem*{min_w}{Theorem \ref{min_w}}
\begin{min_w}
If $\CCC$ is a hereditary class of graphs of bounded vertex degree or has an excluded minor then it does not contain a minimal class.
\end{min_w}

\newtheorem*{min_nested_2}{Theorem \ref{min_nested_2}}
\begin{min_nested_2}
If $\RRR^{\alpha}$ is a path-star hereditary class of graphs defined by a nested word $\alpha$ then it does not contain a minimal class.
\end{min_nested_2}

This suggests the following conjecture:

\begin{conj}
Sparse hereditary graph classes of unbounded tree-width do not contain a minimal class of unbounded tree-width.
\end{conj}

%
%
%
%
%
\section{Preliminaries} \label{prelim}

\subsection{Graphs - General}

A graph $G=(V,E)$ is a pair of sets, vertices $V=V(G)$ and edges $E=E(G)\subseteq [V]^2$. Unless otherwise stated, all graphs in this paper are simple, i.e., undirected, without loops or multiple edges. The number of vertices in a graph $G$ is denoted $|G|$ and the number of edges $\|G\|$. The \emph{degree} $d_G(v)=d(v)$ of a vertex $v$ is the number of edges at $v$. 

If vertex $u$ is adjacent to vertex $v$ we write $u \sim v$. We denote $N(v)$ as the neighbourhood of a vertex $v$, that is, the set of vertices adjacent to $v$. A set of vertices is \emph{independent} if no two of its elements are adjacent and is a \emph{clique} if all the vertices are pairwise adjacent. We denote a clique with $r$ vertices as $K_r$ and an independent set of $r$ vertices as $\overline{K_r}$. A graph is \emph{bipartite} if its vertices can be partitioned into two independent sets, $V_1$ and $V_2$, and is \emph{complete bipartite} if, in addition, each vertex of $V_1$ is adjacent to each vertex of $V_2$. We denote this by $K_{r,s}$ where $|V_1|=r$ and $|V_2|=s$. 

A \emph{tree} is a graph in which any two vertices are connected by exactly one path and a \emph{forest} is a collection of disjoint trees. A \emph{star} $S_k$ is the complete bipartite graph $K_{1,k}$: a tree with one internal vertex, which we will refer to as the \emph{star-vertex}, and $k$ \emph{leaves}.

A \emph{$k$-cycle} is a closed path with $k$ vertices. A $k$-cycle  has a \emph{chord} if two of its $k$ vertices are joined by an edge which is not itself part of the cycle.  A \emph{hole} is a chordless cycle of length at least $4$.

An \emph{$(m \times n)$-wall} is a graph whose edges are visually equivalent to the mortar lines of a stretcher-bonded clay brick wall with $m$ rows of bricks each of which is $n$ bricks long. More precisely, we can define the wall $W_{m \times n}=(V,E)$ using a square grid of the usual $(x,y)$ Cartesian coordinates.
\begin{align*}
&V= \{(x,y): 0 \le x \le 2n+1, 0 \le y \le m\}\\
&E_{H}= \{(x,y)(x+1,y): (x,y),(x+1,y) \in V\}\\
&E_V= \{(x,y)(x,y+1): (x,y),(x,y+1) \in V, x+y=0 \text{ }(\text{mod }2)  \}\\
&E= E_{H} \cup E_V. 
\end{align*}
See example of $W_{4 \times 4}$ in Figure \ref{fig:wall1}.

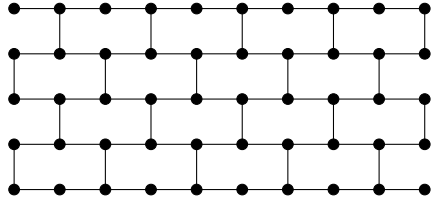
\begin{figure}\centering

\begin{tikzpicture}[scale=0.6,
	vertex2/.style={circle,draw,minimum size=6,fill=black},]

	\foreach \i in {0,...,9}
		\foreach \j in {0,...,4}
			\node at (\i,\j) {};
	
	\draw (0,0)--(9,0);\draw (0,1)--(9,1);\draw (0,2)--(9,2);
	\draw (0,3)--(9,3);\draw (0,4)--(9,4);	
	
	\draw (0,0)--(0,1);\draw (0,2)--(0,3);
	\draw (1,1)--(1,2);\draw (1,3)--(1,4);
	\draw (2,0)--(2,1);\draw (2,2)--(2,3);
	\draw (3,1)--(3,2);\draw (3,3)--(3,4);
	\draw (4,0)--(4,1);\draw (4,2)--(4,3);
	\draw (5,1)--(5,2);\draw (5,3)--(5,4);
	\draw (6,0)--(6,1);\draw (6,2)--(6,3);
	\draw (7,1)--(7,2);\draw (7,3)--(7,4);
	\draw (8,0)--(8,1);\draw (8,2)--(8,3);
	\draw (9,1)--(9,2);\draw (9,3)--(9,4);
	
\end{tikzpicture}
\caption{The $4 \times 4$ wall $W_{4 \times 4}$}
	\label{fig:wall1}

\end{figure}

The \emph{line graph} of a graph $G=(V,E)$ is the graph $L(G)$ with vertex set  $E$ where two vertices in $L(G)$ are adjacent if and only if they are incident as edges in $G$. A graph $H$ is a \emph{subdivision} of a graph $G$ if $H$ can be obtained from $G$ by inserting new (degree $2$) vertices on some of the edges of $G$.

We will use the notation $H \le_I G$ to denote graph $H$ is an \emph{induced subgraph} of graph $G$, meaning $V(H) \subseteq V(G)$ and two vertices of $V(H)$ are adjacent in $H$ if and only if they are adjacent in $G$. We will denote the subgraph of $G=(V,E)$ induced by the set of vertices $U \subseteq V$ by $G[U]$. If a graph $G$ does not contain an induced subgraph isomorphic to $H$ we say that $G$ is \emph{$H$-free}.

An \emph{embedding} of a graph $H$ in a graph $G$ is an injective map $\phi:V(H) \rightarrow V(G)$ such that the subgraph of $G$ induced by the vertices $\phi(V(H))$ is isomorphic to $H$. In other words, $vw \in E(H)$ if and only if $\phi(v) \phi(w) \in E(G)$. If $H$ is an induced subgraph of $G$, then this can be witnessed by one or more embeddings. 

A class of graphs $\CCC$ is \emph{hereditary} if it is closed under taking induced subgraphs, that is $G \in \CCC$ implies $H \in \CCC$ for every induced subgraph $H$ of $G$. It is well known that for any hereditary class $\CCC$ there exists a unique (but not necessarily finite) set of minimal forbidden graphs $\{F_1, F_2, \dots\}$ such that $\CCC= \Free(F_1, F_2, \dots)$ (i.e., every  graph $G \in \CCC$ is $F_i$-free for $i=1,2,\dots$). We will use the notation $\CCC \subseteq \RRR$ to denote that $\CCC$ is a  \emph{hereditary subclass} of a hereditary graph class $\RRR$ ($\CCC \subsetneq \RRR$ for a proper subclass).

%
%
%
%
%
%
\subsection{Sparsity and Density}\label{sparse}

Recall from Section \ref{Sect:Intro} that a \emph{$k$-uniformly sparse} graph is one whose edges can be covered by at most $k$ forests. The following theorems are useful:

\begin{thm} [Nash-Williams, 1964 \cite{nash-williams:decompositions:}]\label{thm_nash}
The edges of a graph $G=(V,E)$ can be covered by at most $k$ forests if and only if $\|G[U] \| \le k(|U|-1)$ for every non-empty set $U \subseteq V$.
\end{thm}

\begin{thm} [Kostochka, 1982 \cite{diestel:graph-theory5:}] \label{thm_kostochka}
There exists a constant $c \in \mathbb{R}$ such that, for every $t \in \mathbb{N}$, every graph $G$ of average degree $d(G) \ge ct\sqrt{\log{t}}$ contains $K_t$ as a minor. 
\end{thm}

A \emph{dense} hereditary graph class is one that is not $k$-uniformly sparse for some $k \in \mathbb{N}$, or in other words, by Theorem  \ref{thm_nash}, for any $k$ there is a graph in the class that has average degree greater than $k$. By Theorem  \ref{thm_kostochka}, for any $t \in \mathbb{N}$ we can set $k \ge ct\sqrt{\log{t}}$ so that the class contains a graph with a $K_t$ minor.

%
%
%
%
%
%
\subsection{Tree-width}

Let $G=(V,E)$ be a graph , $T$ a tree, and let $\mathcal{V}=(V_t)_{t\in T}$ be a family of vertex sets $V_t \subseteq V(G)$ (called \emph{bags}) indexed by the vertices $t$ of $T$. The pair $(T,\mathcal{V})$ is called a \emph{tree-decomposition} of $G$ if it satisfies the following three conditions:
\begin{enumerate}
	\item Every vertex of $G$ is in at least one of the bags $V_t$,
	\item If $(u,v) \in E$, then $u$ and $v$ are together in some bag,
	\item for all $v \in V$, the graph induced by the bags containing $v$ is connected in $T$.
\end{enumerate}

The \emph{width} of a tree-decomposition is the maximum bag size minus $1$. The \emph{tree-width} of $G$, denoted by $\tw(G)$,  is the least width of any tree-decomposition of $G$.

\begin{lemma}[\cite{diestel:graph-theory5:}]\label{lem_minor}
If $H$ is a minor of $G$ then $\tw(H) \le \tw(G)$.
\end{lemma}

It is easy to show that $\tw(K_t) = t-1$ for all positive integers $t \ge 2$. Theorems \ref{thm_nash} and \ref{thm_kostochka} tell us that all dense hereditary graph classes contain a graph with a $K_t$ minor for any positive integer $t \ge 2$ and so by Lemma \ref{lem_minor} have unbounded tree-width. 

Similarly, recalling the definition of $t$-sails (see Definition \ref{def_tsail}), we have:

\begin{lemma}\label{lem_sail}
If $G$ is a $t$-sail for positive integer $t \ge 2$  then $\tw(G) \ge t-1$.
\end{lemma}
\begin{proof}
Contracting star-vertex $s_i$ with the vertices of path $P_i$ for $1 \le i \le t$ gives a $K_{t}$-minor, then using Lemma \ref{lem_minor} and the fact that $\tw(K_t) = t-1$ we have $\tw(G) \ge t-1$.  
\end{proof}

It follows that hereditary graph classes containing arbitrarily large $t$-sails have unbounded tree-width. 

%
%
%
%
%
%
\subsection{Clique-width}\label{Clique_width}

\emph{Clique-width} is a graph width parameter introduced by Courcelle, Engelfriet and Rozenberg in the 1990s~\cite{courcelle:handle-rewritin:}. The clique-width of a graph is denoted $\cw(G)$ and is defined as the minimum number of labels needed to construct $G$ by means of four graph operations -- see \cite{courcelle:upper-bounds-to:}.

Bounded tree-width always implies bounded clique-width:

\begin{thm} [Courcelle and Olariu, \cite{courcelle:upper-bounds-to:}] \label{thm_cw}
For every graph $G$, $\cw(G) \le 2^{\tw(G)+1}+1$.
\end{thm}

However, bounded clique-width does not always imply bounded tree-width. For example, the class of complete graphs has bounded clique-width but unbounded tree-width. Combining results from \cite{courcelle:upper-bounds-to:} with results from Gurski and Wanke \cite{gurski-wanke:tree-width-clique-width-Knn:} gives us certain graph classes for which tree-width and clique-width are either both bounded or both unbounded:

\begin{thm} [\cite{courcelle:upper-bounds-to:, gurski-wanke:tree-width-clique-width-Knn:}] \label{thm_tw_cw}
If $\CCC$ is a collection of graphs such that every graph $G \in \CCC$ either
\begin{enumerate}[label=(\roman*)]
	\item has bounded vertex degree of no more than a constant $\Delta$, or 
	\item excludes a fixed graph $H$ as a minor, or
	\item has bounded arboricity,
	\end{enumerate}
then $\CCC$ has unbounded clique-width if and only if it has unbounded tree-width.
\end{thm}

%
%
%
%
%
%

\section{Nested Words} \label{Sect:Words}

In this Section we introduce \emph{nested words}, describe their characteristics and provide a range of examples and methods of generation to demonstrate that not only do they exist but are, in fact, quite common. These words are central to the analysis in Sections \ref{Sect:Nested} and \ref{Sect_min} where we show that path-star classes defined by such words  have a number of interesting features: they are $KKW$-free, $t$-sails are the basic objects obstructing bounded tree-width, they are infinitely defined and do not contain a minimal class of unbounded tree-width.

%
%
%
%
\subsection{Symbolic sequences (words)}\label{words}

We refer to a (finite or infinite) sequence of letters chosen from a (finite or infinite) alphabet as a \emph{word}. In this paper we use the natural numbers as letters to create infinite words that are used to define graph classes. We denote by $\alpha_i$ the $i$-th letter of the word $\alpha$. A \emph{factor} of $\alpha$ is a contiguous subword $\alpha_{[i,j]}$ being the sequence of letters from the $i$-th to the $j$-th letter of $\alpha$. The \emph{length} of a word (or factor) is the number of letters the word contains.

Given a word $\alpha$ over an alphabet $\AAA$, and a sub-alphabet $\SSS \subset \AAA$, the \emph{subword} of $\alpha$ restricted to $\SSS$ is the word derived from $\alpha$ by deleting all letters not in $\SSS$ and concatenating the remaining factors in the same order as they appear in $\alpha$. We denote this subword as $\alpha^{\SSS}$.

We define \emph{branched} and \emph{nested} words as follows:

\begin{defn}\label{def_branched}
A word $\alpha$ over alphabet $\AAA=\mathbb{N}$ is \emph{branched} if $\AAA$ can be partitioned into a finite \emph{base set} $\BBB$ and ordered (by the order inherited from $\mathbb{N})$ \emph{branch sets} $\{\HHH_1, \HHH_2, \dots\}$ which may be finite or infinite, such that:
\begin{itemize} 
\item A base letter can appear after any other letter.
\item The first letter in a branch $\HHH_i$ can only appear after a base letter.
\item Any other letter in a branch $\HHH_i$ can only appear  after the letter preceding it in the $\HHH_i$ order.  
\end{itemize}
\end{defn}

A maximal factor of branched word $\alpha$ containing no base letters is called a \emph{branch} of $\alpha$. Thus, a branch in $\alpha$ is preceded by a base letter unless it begins on the first letter of $\alpha$ and is succeeded by a base letter unless it ends on the last letter of $\alpha$. The letters in the branch come from one branch set, say $\HHH_i$, and appear, starting with the first letter in $\HHH_i$, in the defined order.

\begin{defn}\label{def_nested}
Let $\alpha$ be an infinite branched word over an infinite alphabet $\AAA$ with respect to base set $\BBB_{\alpha}$ and branch sets $\{\HHH^{\alpha}_1, \HHH^{\alpha}_2, \dots\}$ with the property that each letter in $\AAA$ appears an infinite number of times in $\alpha$. Then $\alpha$ is \emph{$\mathfrak{b}$-nested} if there exists a fixed positive integer $\mathfrak{b}$ such that  any subset $\SSS \subseteq \AAA$ can be partitioned into a base set $\BBB_{\SSS}$ and ordered branch sets $\{\HHH^{\SSS}_1, \HHH^{\SSS}_2, \dots\}$, such that  
\begin{itemize} 
\item $|\BBB_{\SSS}| \le \mathfrak{b}$, 
\item $\alpha^{\SSS}$ is a branched word with base set $\BBB_{\SSS}$, and
\item the branch sets of $\alpha^{\SSS}$ are (possibly empty) subsets of the branch sets of $\alpha$ (i.e $\HHH^{\SSS}_i \subseteq \HHH^{\alpha}_{i}$ with the same ordering, for each $i$).  
\end{itemize}
\end{defn}

\begin{example} \label{ex_branch_notnest}
Consider the words $\sigma^1$, $\sigma^2$ and $\alpha$ where $\sigma^1$ is all $1$s, $\sigma^2$ is the word 
\[2323432345432345654323456765432345678 \dots\] and $\alpha$ is the word whose odd letters are $\sigma^1$ and even letters are $\sigma^2$, so that: 
 
\[\alpha=\textcolor{blue}{1}\, 2\, \textcolor{blue}{1}\, 3\, \textcolor{blue}{1}\, 2 \, \textcolor{blue}{1}\, 3\, \textcolor{blue}{1} \,4\, \textcolor{blue}{1}\, 3\, \textcolor{blue}{1}\, 2\, \textcolor{blue}{1}\, 3\, \textcolor{blue}{1}\, 4\, \textcolor{blue}{1}\, 5\, \textcolor{blue}{1} \,4\, \textcolor{blue}{1}\, 3\, \textcolor{blue}{1}\, 2\, \textcolor{blue}{1}\, 3\, \textcolor{blue}{1}\, 4\, \textcolor{blue}{1}\, 5\, \textcolor{blue}{1}\, 6 \dots\]

$\alpha$ is branched with $\BBB=\{1\}$ (in blue) and branch sets $\HHH_1=\{2\}$, $\HHH_2=\{3\}$, $\HHH_3=\{4\}$ etc. It is an infinite word with an infinite alphabet $\AAA=\mathbb{N}$ where each letter appears an infinite number of times. However, it is not nested, since letting $\SSS=\AAA \setminus 1$ then $\alpha^{\SSS}=\sigma^2$  is not a branched word since there is no finite base set $\BBB_{\SSS}$ that can partition $\alpha^{\SSS}$ into one-letter branches that are subsets of the branch sets of $\alpha$, which contradicts the third bullet-point in the definition of $\mathfrak{b}$-nested. 
\end{example}

We refer generally to \emph{nested words} when referring to a collection of $\mathfrak{b}$-nested words for some unspecified $\mathfrak{b}$. Some examples of nested words are given in Sections \ref{word_arith} and \ref{word_power}.

\subsection{Examples : Arithmetic nested words}\label{word_arith}

The following nested words have infinite branch sets.

\begin{itemize}
\item {$\alpha(1)$}: One branch set with increasing  branch sizes (base letters shown in blue, gaps in word used to ease parsing): 
\[\alpha_1=\textcolor{blue}{1}2\,\textcolor{blue}{1}23\,\textcolor{blue}{1}234\,\textcolor{blue}{1}2345\,\textcolor{blue}{1}23456\,\textcolor{blue}{1}234567\,\textcolor{blue}{1}2345678\,\textcolor{blue}{1}23456789\,\textcolor{blue}{1}\dots\]
\item {$\alpha(2)$}: Two branch sets (residue classes mod $2$) with increasing  branch sizes: 
\[\alpha_2=\textcolor{blue}{1}\,\textcolor{blue}{2}\,\textcolor{blue}{1}3\,\textcolor{blue}{2}4\,\textcolor{blue}{1}35\,\textcolor{blue}{2}46\,\textcolor{blue}{1}357\,\textcolor{blue}{2}468\,\textcolor{blue}{1}3579\,\textcolor{blue}{2}468\,10\dots\] 
\item {$\alpha(k)$}: $k$ branch sets (residue classes mod $k$) with increasing branch sizes: 
\begin{align*}
\alpha_k=&\textcolor{blue}{1}\,\textcolor{blue}{2}\dots\,\textcolor{blue}{k}\,\,\,\textcolor{blue}{1}\,(k+1)\,\,\textcolor{blue}{2}\,(k+2)\dots\,\textcolor{blue}{k}\,(2k)\,\dots\,\\
&\textcolor{blue}{1}\,(k+1)\,\,(2k+1)\,\,\textcolor{blue}{2}\,\,(k+2)\,\,(2k+2)\dots\,\dots\,\textcolor{blue}{k}\,\,(2k)\,\,(3k)\,\dots 
\end{align*} 
\end{itemize}

\begin{lemma}\label{alpha_nested}
$\alpha(k)$ is a nested word for all $k \in \mathbb{N}$.
\end{lemma}
\begin{proof}
Observe $\alpha(k)$ is branched with $\AAA=\mathbb{N}, \BBB=\{1,2, \dots, k\}, \HHH_1=\{k+1,2k+1,\dots, nk+1, \dots\}, \HHH_2=\{k+2,2k+2,\dots, nk+2, \dots\}, \dots, \HHH_k=\{2k,3k,\dots, nk, \dots\}$.

Let $\SSS$ be any subset of $\mathbb{N}$. Define $\SSS_i=\{n \in \SSS: n\equiv i \mod k\}$ for $1 \le i \le k$, $m_i=\min(\SSS_i)$ for $1 \le i \le k$, $\BBB_{\SSS}= \cup_{i=1}^k m_i$ and $\HHH_{\SSS i}= \SSS_i \setminus m_i$.  Then $\alpha(k)^{\SSS}$ is branched with base $\BBB_{\SSS}$, $|\BBB_{\SSS}| \le k$ and branch sets $\HHH_{\SSS i}$. Hence, $\alpha(k)$ satisfies the conditions of Definition \ref{def_nested}.
\end{proof}

\subsection{Examples: Power nested words}\label{word_power}

The following nested words have an infinite number of single letter branch sets.

\subsubsection{\texorpdfstring{$q$}{q{}}-ary representation}

For natural numbers $q$ and $n$, let $n_q$ be the representation of $n$ in $q$-ary. Let $k$ be the number of trailing zeros of $n_q$ and let $j$ be the first non-zero digit from the right. Alternatively, there exist unique $j,k \in \mathbb{N}$ and $m \in \mathbb{N}_0$ such that $n=jq^k+ mq^{k+1}$ ($1 \le j \le q-1$). 

We define infinite \emph{power} words $\kappa(q)$ such that the $n$-th letter $\kappa(q)_n=i$ where $i=k(q-1)+j$ (e.g. for $q=3$ we have $\kappa(3)_{45}=6$ because $45$ base $3$ is $1200$, so $k=2$, $j=2$ and $i=2(3-1)+2=6$).

This construction gives us:
\begin{align*}
&\kappa(2)=\textcolor{blue}{1}2\textcolor{blue}{1}\,3\,\textcolor{blue}{1}2\textcolor{blue}{1}\,4\,\textcolor{blue}{1}2\textcolor{blue}{1}3\textcolor{blue}{1}2\textcolor{blue}{1}\,5\,\textcolor{blue}{1}2\textcolor{blue}{1}3\textcolor{blue}{1}2\textcolor{blue}{1}4\textcolor{blue}{1}2\textcolor{blue}{1}3\textcolor{blue}{1}2\textcolor{blue}{1}\,6\dots\\
&\kappa(3)=\textcolor{blue}{12}\,3\,\textcolor{blue}{12}\,4\,\textcolor{blue}{12}\,5\,\textcolor{blue}{12}3\textcolor{blue}{12}4\textcolor{blue}{12}\,6\,\textcolor{blue}{12}3\textcolor{blue}{12}4\textcolor{blue}{12}\,7\,\textcolor{blue}{12}3\textcolor{blue}{12}4\textcolor{blue}{12}5\dots\\
&\kappa(4)=\textcolor{blue}{123}\,4\,\textcolor{blue}{123}\,5\,\textcolor{blue}{123}\,6\,\textcolor{blue}{123}\,7\,\textcolor{blue}{123}4\textcolor{blue}{123}5\textcolor{blue}{123}6\textcolor{blue}{123}\,8\,\textcolor{blue}{123}4\dots\\
&\kappa(5)=\textcolor{blue}{1234}\,5\,\textcolor{blue}{1234}\,6\,\textcolor{blue}{1234}\,7\,\textcolor{blue}{1234}\,8\,\textcolor{blue}{1234}\,9\,\textcolor{blue}{1234}5\textcolor{blue}{1234}6\textcolor{blue}{1234}7\dots
\end{align*}
 
(Note that $\kappa(2)$ has previously appeared in the literature in the context of dense graphs \cite{lozin:well-quasi-ordering-does-not:} and in the context of sparse graphs in \cite{Bonamy:sparse_graphs_log_treewidth:}.)  
 
These words can also be generated by a recurrence relation. For example,  $\kappa(2)$ can be generated by the recurrence relation  $\kappa(2)=\lim_{n \rightarrow \infty} \kappa(2)^n$ where $\kappa(2)^1 = 1$ and for $n>1$,  $\kappa(2)^n=\kappa(2)^{n-1}(n)\kappa(2)^{n-1}$. 
The first four iterates are as follows:
\begin{align*}
\kappa(2)^1    = 1 , \kappa(2)^2	= 1\,2\,1 , \kappa(2)^3	= 121\,3\,121 ,\kappa(2)^4	= &1213121\,4\,1213121
\end{align*}	

In general: 

\begin{prop}\label{prop_kappa}
$\kappa(q)$ can be generated by the recurrence relation $\kappa(q)=\lim_{n \rightarrow \infty} \kappa(q)^n$ where $\kappa(q)^1 = 123\dots\, q-1$ and for $n>1$,  
\begin{align*}
\kappa(q)^n=&\kappa(q)^{n-1}((n-1)(q-1)+1)\kappa(q)^{n-1}((n-1)(q-1)+2) \dots\\
&\kappa(q)^{n-1}((n-1)(q-1)+(q-2))\kappa(q)^{n-1}(n(q-1))\kappa(q)^{n-1}.
\end{align*}
\end{prop}
\begin{proof}
Notice that by induction, $|\kappa(q)^1|= q-1$, $|\kappa(q)^2|= q^2-1$ and $|\kappa(q)^n|= q^n-1$ and recall that for the $x$-th letter in $\kappa(q)$, where $x=jq^k+mq^{k+1}$, $\kappa(q)_x=k(q-1)+j$.

Therefore, in $\kappa(q)^n$ the letters in the interval  $[jq^{n-1}+1, (j+1)q^{n-1}-1]$ for $0 \le j \le q-1$ are identical to $\kappa(q)^{n-1}$ and the letter in location $jq^{n-1}$ for $0 \le j \le q-1$ is $(n-1)(q-1)+j$. The recurrence relation follows.
\end{proof}

\begin{lemma}\label{kappa_nested}
$\kappa(q)$ is a nested word for all $q \in \mathbb{N}$.
\end{lemma}
\begin{proof}
Observe $\kappa(q)$ is branched with $\AAA=\mathbb{N}, \BBB=\{1,2, \dots, q-1\}, \HHH_1=\{q\}, \HHH_2=\{q+1\}, \dots, \HHH_k=\{q+k-1\}, \dots$.

Let $\SSS=\{x_1,x_2,\dots\}$ be a subset of $\mathbb{N}$,  where $x_1 < x_2 \dots$. Let $n$ be a position in $\kappa(q)$ with the letter $x_1$ and let $j,k,m$ be the unique integers such that $n=jq^k+ mq^{k+1}$ and $x_1= k(q-1)+j$ as described above. Further, let $x_i \ge x_1$ be the highest element of $\SSS$ such that  $x_i < kq$. We claim that $\kappa(q)$ is nested over $\SSS$ with base $\BBB_{\SSS}=\{x_1, x_2, \dots x_i\}$ of maximum size $q-1$ with single letter branch sets $\HHH^{\SSS}_ 1=\{x_{i+1}\}, \dots , \HHH^{\SSS}_t=\{x_{i+t}\}, \dots$.

For any positive integer $t$ the letter immediately preceding or succeeding  an appearance of $x_{i+t}$  in $\kappa(q)^{\SSS}$ is either $x_1$ or $x_i$. Hence, $\kappa(q)^{\SSS}$ is branched with base $\BBB_{\SSS}$, $|\BBB_{\SSS}|\le q-1$ and branch sets $\HHH^{\SSS}_t$, and $\kappa(q)$ satisfies the conditions of Definition \ref{def_nested}.
\end{proof}

\subsubsection{Fibonacci representation}

The well-known Fibonacci sequence of numbers is defined recursively as $F_0=0, F_1=1, F_n=F_{n-1}+ F_{n-2} \text{ for } n \ge 2$ so that
\[F_1=1,F_2=1,F_3=2,F_4=3,F_5=5,F_6=8,F_7=13,F_8=21, \dots\]

The Fibonacci representation of a number is a sequence of $0$s and $1$s, rather like binary, except that a $1$ in position $k$ counting from the right represents $F_{k+1}$ instead of $2^{k-1}$. Note that to avoid having a repeated $1$ in the sequence we start with $F_2$.  So, for example, $101101$ represents $2^5+2^3+2^2+2^0=32+8+4+1=45$ in binary but $F_7+F_5+F_4+F_2=13+5+3+1=22$ as a Fibonacci representation. In general, a Fibonacci representation is not unique. However, Zeckendorf \cite{Zeckendorf:Fibonacci-representation:} showed that every positive integer can be represented uniquely as the sum of non-consecutive Fibonacci numbers, so we will only use the Zeckendorf representation here.

We define the infinite word $\eta$ such that $\eta_n=i$ if the first $1$ in the Fibonacci representation of $n$ (from the right) appears in position $i$ (e.g. $n=45$ which has Fibonacci representation $10010100$ gives $\eta_{45}=3$). Thus:
 
\[\eta=\textcolor{blue}{12}\,3\,\textcolor{blue}{1}\,4\,\textcolor{blue}{12}\,5\,\textcolor{blue}{12}3\textcolor{blue}{1}\,6\,\textcolor{blue}{12}3\textcolor{blue}{1}4\textcolor{blue}{12}\,7\,\textcolor{blue}{12}3\textcolor{blue}{1}4\textcolor{blue}{12}5\textcolor{blue}{12}3\textcolor{blue}{1}\,8\dots\] 

Also observe:

\begin{prop}\label{prop_eta}
$\eta$ can be generated by the recurrence relation $\eta=\lim_{n \rightarrow \infty} \eta^n$ where $\eta^1 = 1$, $\eta^2 = 12$  and for $n>2$,  $\eta^n=\eta^{n-1}(n)\eta^{n-2}$.
\end{prop}
\begin{proof}
Notice that by induction, $|\eta^1|= 1$, $|\eta^2|= 2$ and $|\eta^n|= F_{n+2}-1$.

By definition, the first $F_{n+1}-1$ letters of $\eta^n$ must be $\eta^{n-1}$.  Equally, the first $1$ in the representation of $F_{n+1}$ is in position $n$ so the $F_{n+1}$-th letter of $\eta^n$ is $n$. Also, the letters in the interval  $[F_{n+1}+1, F_{n+2}-1]$ of length $F_{n}-1$ are identical to the letters in the interval $[1, F_{n}-1]$, and the recurrence relation follows.
\end{proof}

\begin{lemma}\label{eta_nested}
$\eta$ is a nested word.
\end{lemma}
\begin{proof}
Observe $\eta$ is branched with $\AAA=\mathbb{N}, \BBB=\{1,2\}, \HHH_1=\{3\}, \HHH_2=\{4\}, \dots, \HHH_k=\{k+2\}, \dots$.

Let $\SSS=\{x_1,x_2,\dots\}$ be a subset of $\mathbb{N}$,  where $x_1 < x_2 \dots$, with $\BBB_{\SSS}=\{x_1, x_2\}$, $\HHH^{\SSS}_ 1=\{x_{3}\}, \dots , \HHH^{\SSS}_t=\{x_{t+2}\}, \dots$. Then, from the Zeckendorf representation, for any $t \ge 3$ the letter immediately preceding or succeeding  an appearance of $x_t$  in $\eta^{\SSS}$ is either $x_1$ or $x_2$. Hence, $\eta^{\SSS}$ is branched with base $\BBB_{\SSS}$, $|\BBB_{\SSS}|\le 2$ and branch sets $\HHH^{\SSS}_t$, and $\eta$ satisfies the conditions of Definition \ref{def_nested}.
\end{proof}

\subsection{Generating new nested words}\label{word_gen}

Other nested words can be generated by similar recurrence relations to those given in Propositions \ref{prop_kappa} and \ref{prop_eta}, although not all such relations give nested words (for instance, if the recurrence relation does not result in a word over an infinite alphabet where each letter repeats an infinite number of times). 

If $\alpha$ is a nested word and $\LLL$ a finite collection of letters (that may or may not be letters in $\AAA$, the alphabet of $\alpha$) then inserting an arbitrary number of letters from $\LLL$ into arbitrary positions in $\alpha$ creates a new word that is also nested, since we can add the finite number of letters in $\LLL$ to the base $\BBB$. Hence, nested words are not rare.

\begin{prop}
There are uncountably many distinct nested words.
\end{prop}
\begin{proof}
Let $\alpha$ be a nested word and $\beta$ an infinite (non-nested) binary word. Interlace the letters of $\alpha$ and $\beta$ to create a new word $\gamma$, so that the even letters of $\gamma$ are $\alpha$ and the odd letters $\beta$. $\gamma$ is nested since we can just add the two letters of $\beta$ to the base of $\alpha$ to create a finite base for $\gamma$. There are uncountably many distinct binary words which gives the result.
\end{proof}

However, the existence of a nested subword $\beta$ in $\alpha$ is not sufficient to make $\alpha$ a nested word. The subword $\alpha \setminus \beta$ may have an unbounded base and contain elements that contradict our desired characteristics - see Section \ref{Sect:KKW_free}.

%
%
%
%
%
%
\section{Path-star hereditary graph classes and \texorpdfstring{$t$}{t{}}-sails} \label{Sect:Path_star}

A convenient way to define a family of hereditary path-star class as described in the introduction is to use an infinite word so that the $i$-th letter in the word indicates the star that connects to the $i$-th vertex in the path. We assume that all leaves of the stars embed in the path since non-embedding leaves have no effect on tree-width, i.e., if $G$ is a finite path-star graph and $H$ is isomorphic to $G$ except for the removal of all non-embedding leaves then $\tw(H)=\tw(G)$.
 
Let $\alpha$ be an infinite word over the alphabet $\mathcal{A}=\mathbb{N}$. We denote the path $P=(V_P,E_P)$ with vertices  $V_P=\{p_j:j \in \mathbb{N}\}$ and edges $E_P=\{(p_j,p_{j+1}): j \in \mathbb{N}\}$. The star-vertices are denoted $V_S=\{s_i:i \in\mathbb{N}\}$ and star edges $E_{S}=\{(p_j,s_{\alpha_j}): j \in \mathbb{N}\}$.

\begin{defn}\label{def_path_star_class}
We define an \emph{infinite path-star graph} $R^\alpha=(V,E)$ where $V=V_P \cup V_S$ and $E=E_P \cup E_S$ (see example in Figure \ref{fig:path_star1}). We define the corresponding \emph{path-star class} $\RRR^\alpha$ to be the finite induced subgraphs of $R^\alpha$.
\end{defn}

Any graph $G \in \RRR^\alpha$ can be witnessed by an embedding $\phi(G)$ into the infinite graph $R^\alpha$. To simplify the presentation we will associate $G$ with a particular embedding in $R^\alpha$ depending on the context.

To avoid confusion when referring to different types of path, we will refer to the \emph{class-path} when referring to the (infinite) path of the path-star class, or a \emph{path component} when referring to a finite section of it.  A path component induced by the vertices $\{p_j, p_{j+1}, \dots, p_{j+k}\}$ we denote $I_{[j,j+k]}$. We use the shorthand \emph{$m$-path-vertex} for a vertex in the class-path corresponding to the letter $m$ in $\alpha$. 

In addition, if $\alpha$ is a nested word over alphabet $\AAA$, and $\alpha^{\SSS}$ a nested subword restricted to the sub-alphabet $\SSS$, then we refer to \emph{base star-vertices} and \emph{base path-vertices} for vertices that correspond to base letters in $\SSS$ and \emph{branch star-vertices} and \emph{branch path-vertices} for vertices that correspond to branch letters in $\SSS$. These vertices will depend on the choice of $\SSS$.

\begin{figure}\centering

\begin{tikzpicture}[scale=0.75,
	vertex2/.style={circle,draw,minimum size=9,fill=blue},]
	
	\foreach \i in {1,...,20} \node (p\i)[label={below:$p_{\i}$}] at (\i-10,0) {};
	\foreach \j in {1,...,19}{\draw (\j-10,0)--(\j-9,0);}
	
	\node (s1) [label={above:$s_1$}] at (-6,5) {};
	\node (s2) [label={above:$s_2$}] at (-2,5) {};
	\node (s3) [label={above:$s_3$}] at (2,5) {};
	\node (s4) [label={above:$s_4$}] at (6,5) {};
	\node (s5) [label={above:$s_5$}] at (10,5) {};

	\draw (s1) -- (p1);\draw (s1) -- (p3);\draw (s1) -- (p5); 
	\draw (s1) -- (p7);\draw (s1) -- (p9);\draw (s1) -- (p11); 
	\draw (s1) -- (p13);\draw (s1) -- (p15);\draw (s1) -- (p17);
	\draw (s1) -- (p19);
	 
	\draw (s2) -- (p2);\draw (s2) -- (p6);\draw (s2) -- (p10); 
	\draw (s2) -- (p14);\draw (s2) -- (p18);
	
	\draw (s3) -- (p4);\draw (s3) -- (p12);\draw (s3) -- (p20);	
	
	\draw (s4) -- (p8);\draw (s5) -- (p16); 

\end{tikzpicture}

\caption{The first section of path-star graph $R^{\kappa(2)}$}
	\label{fig:path_star1}

\end{figure}
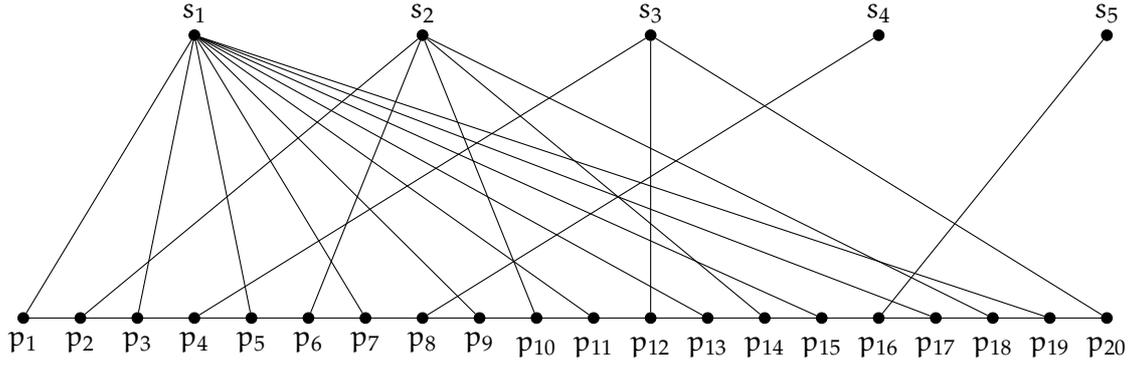

%
%
%
%
%
%
\subsection{Path-star classes with unbounded tree-width and clique-width}

Throughout this section let $\AAA$ be an alphabet and $\alpha$ be an infinite word over $\AAA$. Let $\mathcal{A}^{\alpha} \subseteq \mathcal{A}$ be the set of letters in $\AAA$ that appear an infinite number of times in $\alpha$. That is, these are the letters of $\AAA$ corresponding to the infinite stars in $R^\alpha$.

\begin{thm} \label{unbd_tw}
If $\mathcal{A}^{\alpha}$ is infinite then the graph class $\RRR^\alpha$ has unbounded tree-width and clique-width.
\end{thm}
\begin{proof}
We will show that $\RRR^\alpha$ contains a $t$-sail for all $t$ and thus has unbounded tree-width by Lemma \ref{lem_sail}. As $\RRR^\alpha$ has arboricity two it follows from Theorem \ref{thm_tw_cw} that $\RRR^\alpha$ also has unbounded clique-width.

Let $\mathcal{A}^{\alpha}=\{i_1, i_2, \dots \}$. For any $t \in \mathbb{N}$ we can create a set of $t$ factors of $\alpha$ as follows. 

Let $I_1=\{j\}$, where $j$ is the position of the first occurrence of letter $i_1$ in $\alpha$. For $2 \le k \le t$ let $I_k=[x,y]$ be the next interval beyond $I_{k-1}$ where $\alpha_{I_k}$ contains all of $i_1, \dots, i_k$. Such intervals can always be found because the letters in $\mathcal{A}^{\alpha}$ repeat infinitely in $\alpha$. This gives us a set of $t$ disjoint factors of $\alpha$, $\{\alpha_{I_k}: 1 \le k  \le t\}$.

Defining the vertex set $V^t=\{p_i:i \in \cup_{1\le k \le t} I_k\} \cup \{s_{i_k}: 1 \le k \le t\}$ means $R^{\alpha}[V^t]$ is a $t$-sail and the result follows.
\end{proof}

%
%
%
%
%
%
\subsection{Path-star classes are not subclasses of circle graphs}

A \emph{circle graph} is an intersection graph of finitely many chords of a circle. Circle graphs are a much studied hereditary class, in particular, because the class is vertex-minor closed (for definition see \cite{geelen:the-grid-theorem:}).
Geelen, Kwon, McCarty and Wollan \cite{geelen:the-grid-theorem:} showed that a vertex-minor-closed graph class has bounded clique-width if and only if it excludes a circle graph as a vertex-minor.

More recently, in \cite{hickingbotham:treewidth_circlegraphs:}, the authors describe the unavoidable induced subgraphs of circle graphs with large tree-width. To distinguish the results in this paper from those in \cite{hickingbotham:treewidth_circlegraphs:}  we show that path-star graph classes are not subclasses of circle graphs.

Let $\alpha$ be a word over an alphabet with at least two letters. We will call a factor of $\alpha$ that starts with one letter $i$ and ends with another $j$, with no other occurrences of either letter in the factor, an \emph{($i,j$)-alternance}. If $G$ is a graph in the path-star class $\RRR^{\alpha^{\{1,2\}}}$ induced by the two stars $s_1$ and $s_2$ and a path component, we  show that it is not possible to construct a chord representation of $G$ when the sequence in $\alpha$ corresponding to the path component has more than four  ($1,2$)-alternances, i.e., $G$ is not a circle graph.  For example, the word $11212221112$ alternates $5$ times between $1$ and $2$ and therefore does not represent a circle graph.

We may refer to $G$ by name or by $\alpha$ letter sequence (e.g. $G=1221221$). We will always label the path vertices of $G$ starting with $p_1$ so that $1221221$ has path vertices $p_1, \dots, p_7$.

\begin{figure}\centering
\begin{tikzpicture}[scale=1.20,]
	
	\draw[red,very thick] (45:2.5) arc (45:135:2.5) node[draw=none,fill=none,midway,above] () {\Large{$A$}} ;
	\draw[blue,very thick] (-45:2.5) arc (-45:45:2.5)node[draw=none,fill=none,midway,right] () {\Large{$D$}} ;
	\draw[red,very thick] (-135:2.5) arc (-135:-45:2.5)node[draw=none,fill=none,midway,below] () {\Large{$B$}} ;
	\draw[blue,very thick] (-225:2.5) arc (-225:-135:2.5)node[draw=none,fill=none,midway,left] () {\Large{$C$}} ;
	
	\draw[ultra thick] (-225:2.5) -- (-135:2.5) ;
	\draw[ultra thick] (-45:2.5) -- (45:2.5) ;
	
	\draw[very thick] (165:2.5) -- (100:2.5) ;
	\draw[very thick] (115:2.5) -- (-30:2.5) ;
	\draw[very thick] (205:2.5) -- (60:2.5) ;
	\draw[very thick] (75:2.5) -- (15:2.5) ;
	
	\node (s21)[draw=none,fill=none][label={above:$s_2$}] at (45:2.5) {};
	\node (s22)[draw=none,fill=none][label={below:$s_2$}] at (-45:2.5) {};
	\node (s11)[draw=none,fill=none][label={above:$s_1$}] at (135:2.5) {};
	\node (s12)[draw=none,fill=none][label={below:$s_1$}] at (-135:2.5) {};
	\node (p11)[draw=none,fill=none][label={left:$p_1$}] at (165:2.5) {};
	\node (p12)[draw=none,fill=none][label={above:$p_1$}] at (100:2.5) {};
	\node (p21)[draw=none,fill=none][label={above:$p_2$}] at (115:2.5) {};
	\node (p22)[draw=none,fill=none][label={right:$p_2$}] at (-30:2.5) {};
	\node (p31)[draw=none,fill=none][label={left:$p_3$}] at (205:2.5) {};
	\node (p32)[draw=none,fill=none][label={above:$p_3$}] at (60:2.5) {};
	\node (p41)[draw=none,fill=none][label={above:$p_4$}] at (75:2.5) {};
	\node (p42)[draw=none,fill=none][label={right:$p_4$}] at (15:2.5) {};

	\node (5)[draw=none,fill=none] at (0,-3.5) {$1212$};

\begin{scope}[shift={(7,0)}]

	\draw[red,very thick] (45:2.5) arc (45:135:2.5) node[draw=none,fill=none,midway,above] () {\Large{$A$}} ;
	\draw[blue,very thick] (-45:2.5) arc (-45:45:2.5)node[draw=none,fill=none,midway,right] () {\Large{$D$}} ;
	\draw[red,very thick] (-135:2.5) arc (-135:-45:2.5)node[draw=none,fill=none,midway,below] () {\Large{$B$}} ;
	\draw[blue,very thick] (-225:2.5) arc (-225:-135:2.5)node[draw=none,fill=none,midway,left] () {\Large{$C$}} ;
	
	\draw[ultra thick] (-225:2.5) -- (-135:2.5) ;
	\draw[ultra thick] (-45:2.5) -- (45:2.5) ;
	
	\draw[very thick] (155:2.5) -- (100:2.5) ;
	\draw[very thick] (115:2.5) -- (30:2.5) ;
	\draw[very thick] (195:2.5) -- (75:2.5) ;
	\draw[very thick] (170:2.5) -- (-60:2.5) ;
	\draw[very thick] (-105:2.5) -- (-30:2.5) ;
	\draw[very thick] (-150:2.5) -- (-80:2.5) ;
	
	\node (s21)[draw=none,fill=none][label={above:$s_2$}] at (45:2.5) {};
	\node (s22)[draw=none,fill=none][label={below:$s_2$}] at (-45:2.5) {};
	\node (s11)[draw=none,fill=none][label={above:$s_1$}] at (135:2.5) {};
	\node (s12)[draw=none,fill=none][label={below:$s_1$}] at (-135:2.5) {};
	\node (p11)[draw=none,fill=none][label={left:$p_1$}] at (155:2.5) {};
	\node (p12)[draw=none,fill=none][label={above:$p_1$}] at (100:2.5) {};
	\node (p21)[draw=none,fill=none][label={above:$p_2$}] at (115:2.5) {};
	\node (p22)[draw=none,fill=none][label={right:$p_2$}] at (30:2.5) {};
	\node (p31)[draw=none,fill=none][label={left:$p_3$}] at (195:2.5) {};
	\node (p32)[draw=none,fill=none][label={above:$p_3$}] at (75:2.5) {};
	\node (p41)[draw=none,fill=none][label={left:$p_4$}] at (170:2.5) {};
	\node (p42)[draw=none,fill=none][label={below:$p_4$}] at (-60:2.5) {};
	\node (p51)[draw=none,fill=none][label={below:$p_5$}] at (-105:2.5) {};
	\node (p52)[draw=none,fill=none][label={right:$p_5$}] at (-30:2.5) {};
	\node (p61)[draw=none,fill=none][label={left:$p_6$}] at (-150:2.5) {};
	\node (p62)[draw=none,fill=none][label={below:$p_6$}] at (-80:2.5) {};

	\node (7)[draw=none,fill=none] at (0,-3.5) {$121121$};

\end{scope}
	
\end{tikzpicture}

\caption{Chord representations of circle graphs $1212$ and $121121$}
	\label{fig:pathstar_circle}

\end{figure}
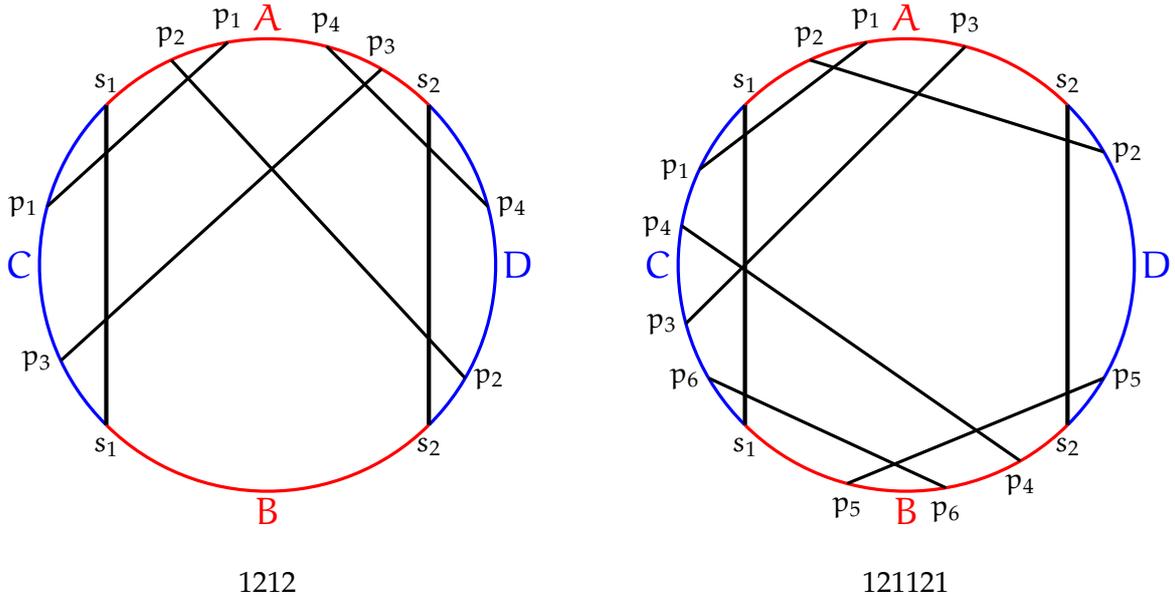

\begin{thm} \label{pathstar_circle}
If there are two letters in the word $\alpha$ that alternate more than four times then the graph class $\RRR^\alpha$ is not a subclass of circle graphs. 
\end{thm}
\begin{proof}
Every graph in $\RRR^{\alpha^{\{1,2\}}}$ is a vertex-minor of a graph in $\RRR^{\alpha}$. As circle graphs are vertex-minor closed, if $\RRR^\alpha$ is a subclass of circle graphs then so is $\RRR^{\alpha^{\{1,2\}}}$. Thus, if we can find a graph in $\RRR^{\alpha^{\{1,2\}}}$ that is not a circle graph then we are done. Also note that there is a $1-1$ correspondence between the ($1,2$)-alternances in $\alpha$ and $\alpha^{\{1,2\}}$. 

Suppose, for a contradiction, that there exists a factor $\beta$ of $\alpha^{\{1,2\}}$ in which the letters $1$ and $2$ alternate more than four times, with the property that the graph $G$ induced by the stars $s_1$ and $s_2$ and the path component corresponding to the factor $\beta$ is a circle graph.

We try to construct a chord representation for $G$ -- see examples in Figure \ref{fig:pathstar_circle}. Without loss of generality, we assume that the first letter of $\beta$ is $1$ and the second $2$.

Note that the chords representing $s_1$ and $s_2$ do not cross,  shown as vertical lines in Figure \ref{fig:pathstar_circle}. Designating the arcs between $s_1$ and $s_2$ $A$ and $B$ (shown in red), and the arcs bounded by $s_1$ and $s_2$ $C$ and $D$ respectively (shown in blue), note that every path vertex adjacent to $s_1$ must be represented by a chord with one end in $C$ and the other in either $A$ or $B$, and similarly for $s_2$, a chord with one end in $D$ and the other in either $A$ or $B$. Therefore,  if $p_1$ and $p_2$ are the two chords representing the first ($1,2$)-alternance then they must cross and both have an end in either $A$ or $B$. Without loss of generality, let them both have an end in $A$. We will show that there cannot be many consecutive alternances happening in the same sector.

Notice that for any $i$, the chords representing $p_{i+2}, p_{i+3}, \dots$ must all be on the same side of the chord representing $p_i$ as none of them can cross this chord. Also notice that if $i<j<k$ and $\alpha_{p_i} =\alpha_{p_j}=1$, $\alpha_{p_k}=2$ then the chord for $p_j$ must be situated on the $s_2$ side of the chord for $p_i$ to accommodate the next alternance (and likewise with $1$ and $2$ reversed).

Suppose that no path-vertex chord has an end in $B$. If $p_3$ is a $1$ (i.e., a second ($1,2$)-alternance) then its chord must cross only $s_1$ and $p_2$. If this is on the 'non-$s_2$' side of $p_1$  then this prevents any further alternance since the path is blocked from star $s_2$ by chord $p_1$. So for there to be a third alternance, $p_3$ must be on the $s_2$ side as shown in the $1212$ example in Figure \ref{fig:pathstar_circle}.   

The chord representing path-vertex $p_4$ cannot cross $p_1$ or $p_2$. Furthermore, if it is on the 'non-$s_1$' side of $p_2$  then this prevents any further alternance since the path is blocked from star $s_1$ by chord $p_2$. Hence, without using arc $B$, we can have at most three ($1,2$)-alternances.

Now suppose that a path-vertex chord may have an end in $B$. We may have at most two alternances through $A$ before switching to $B$, as if we start with three alternances through $A$, as shown in the $1212$ example in Figure \ref{fig:pathstar_circle}, then $p_4$ is blocked from $B$.

If we switch to $B$ after two alternances then $p_4$ is a $1$ with chord ends in $C$ and $B$. It is possible to have at most two alternances through $B$ before we reach $p_6$ which is blocked by $p_4$, as shown in the $121121$ example in Figure \ref{fig:pathstar_circle}, and thereafter no further alternance is possible either via $A$ or $B$.

It follows that the maximum number of alternances possible is four. By assumption, the letters in $\beta$ alternate more than four times and hence we cannot construct a chord representation of $G$, and we have a contradiction.
\end{proof}

%
%
%
%
%
%

\section{Nested path-star hereditary graph classes} \label{Sect:Nested}

We now focus on path-star graph classes defined by nested words.

%
%
%
%
%
%

\subsection{Nested path-star classes are \texorpdfstring{$KKW$}{KKW{}}-free}\label{Sect:KKW_free}

It is quite possible for path-star graph classes to contain a large wall -- see an example in Figure \ref{fig:wall2}. However, we  show that path-star graph classes created from nested words   (see Definition \ref{def_nested}) are $KKW$-free. 

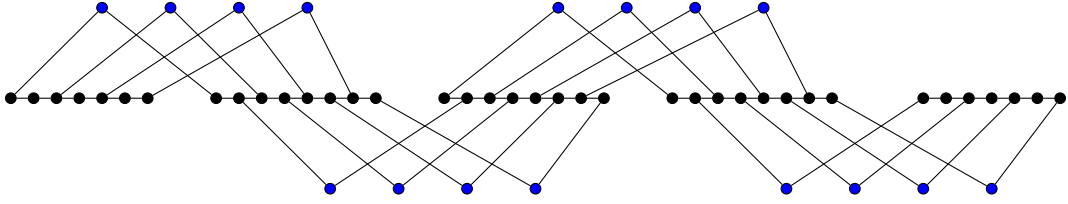
\begin{figure}\centering

\begin{tikzpicture}[scale=0.30,
	vertex2/.style={circle,draw,minimum size=4,fill=blue},]
	
	\foreach \i in {2,3,4,5,6,7,8,11,12,13,14,15,16,17,18,
	21,22,23,24,25,26,27,28,31,32,33,34,35,36,37,38,
	42,43,44,45,46,47,48} \node (p\i) at (\i-10,0) {};
	
	\node [vertex2] (s2)  at (-4,4) {};
	\node [vertex2](s4)   at (-1,4) {};
	\node [vertex2](s6)   at (2,4) {};
	\node [vertex2](s8)   at (5,4) {};
	\node [vertex2](s20)  at (16,4) {};
	\node [vertex2](s22)  at (19,4) {};
	\node [vertex2](s24)  at (22,4) {};
	\node [vertex2](s26)  at (25,4) {};
	
	\node [vertex2](s10)  at (6,-4) {};
	\node [vertex2](s12)  at (9,-4) {};
	\node [vertex2](s14)  at (12,-4) {};
	\node [vertex2](s16)  at (15,-4) {};
	\node [vertex2](s30)  at (26,-4) {};
	\node [vertex2](s32)  at (29,-4) {};
	\node [vertex2](s34)  at (32,-4) {};
	\node [vertex2](s36)  at (35,-4) {};

	\draw (p2)--(p3)--(p4)--(p5)--(p6)--(p7)--(p8);	
	\draw (p11)--(p12)--(p13)--(p14)--(p15)--(p16)--(p17)--(p18);
	\draw (p21)--(p22)--(p23)--(p24)--(p25)--(p26)--(p27)--(p28);
	\draw (p31)--(p32)--(p33)--(p34)--(p35)--(p36)--(p37)--(p38);
	\draw (p42)--(p43)--(p44)--(p45)--(p46)--(p47)--(p48);
		
	\draw (s2) -- (p2);	\draw (s4) -- (p4); 
	\draw (s6) -- (p6); \draw (s8) -- (p8); 
	\draw (s2) -- (p11); \draw (s4) -- (p13); 
	\draw (s6) -- (p15); \draw (s8) -- (p17);
	 
	\draw (s10) -- (p12); \draw (s12) -- (p14); 
	\draw (s14) -- (p16); \draw (s16) -- (p18);
	\draw (s10) -- (p22); \draw (s12) -- (p24); 
	\draw (s14) -- (p26); \draw (s16) -- (p28);
	
	\draw (s20) -- (p21);	\draw (s22) -- (p23); 
	\draw (s24) -- (p25); \draw (s26) -- (p27); 
	\draw (s20) -- (p31); \draw (s22) -- (p33); 
	\draw (s24) -- (p35); \draw (s26) -- (p37);
	
	\draw (s30) -- (p32); \draw (s32) -- (p34); 
	\draw (s34) -- (p36); \draw (s36) -- (p38);
	\draw (s30) -- (p42); \draw (s32) -- (p44); 
	\draw (s34) -- (p46); \draw (s36) -- (p48);

\end{tikzpicture}

\caption{Example of a subdivision of a $4 \times 4$-wall embedded in a path-star graph (star-vertices blue)}
	\label{fig:wall2}

\end{figure}

Observe that if $\alpha$ is a nested word then any connected graph $G$ in $\RRR^{\alpha}$ that does not contain a base-path-vertex contains only star- and path-vertices corresponding to a single branch of $\alpha$, since it is not possible to have a path connecting vertices corresponding to two different branches that does not contain a base-path-vertex. 

\begin{lemma}\label{lem_twostar}
If $\alpha$ is a nested word  and $G$ is a graph in $\RRR^{\alpha}$ that does not contain a base-path-vertex then any induced hole in $G$ must contain exactly two star-vertices.
\end{lemma}
\begin{proof}
Clearly, a hole must contain at least one star-vertex, otherwise it would be a path. Suppose it contained only one star-vertex, say $s_x$, then $\alpha$ would have a factor $x \dots x$ containing no base letter since $G$ has no base-path-vertices.  Since the vertices only correspond to a single branch of $\alpha$ and branch letters appear in branch-order, such a factor does not exist so we have a contradiction. 

Suppose our hole contains three or more star-vertices, say $s_x$, $s_y$ and $s_z$ where $x$, $y$ and $z$ appear in this order in a branch. This requires the three stars to be connected by path segments corresponding to branch sequences $x \dots y $, $x \dots z$ and $y \dots z$ in $\alpha$. As $G$ is single-branch then the sequence  $x \dots z$ must contain the letter $y$. The corresponding path-vertex must be adjacent to $s_y$ creating a chord in the cycle, so it is not a hole. A contradiction.

Therefore, the only possibility is that any hole in $G$ must contain exactly two star-vertices.
\end{proof}

We will call a graph consisting of five 'bricks' of a wall, as shown with numbered vertices in Figure \ref{fig:wall3}, a \emph{$5$-wall}.

Notice that a $5$-wall contains exactly eleven induced chordless cycles or holes -- which we will call $H_1, \dots, H_{11}$ (shown in Figure \ref{fig:wall3}) with vertex sets $V_1, \dots, V_{11}$ respectively.

\begin{figure}\centering

\begin{tikzpicture}[scale=0.9,
	vertex2/.style={circle,draw,minimum size=6,fill=blue},]

	\node   (1) [label={above:$1$}] at (0,3) {};
	\node   (2) [label={above:$2$}] at (1,3) {};
	\node   (3) [label={above:$3$}] at (2,3) {};
	\node   (4) [label={above:$4$}] at (3,3) {};
	\node   (5) [label={above:$5$}] at (4,3) {};
	\node   (6) [label={above:$6$}] at (5,3) {};
	\node   (7) [label={above:$7$}] at (6,3) {};
	
	\node   (8) [label={below:$8$}] at (0,2) {};
	\node   (9) [label={above:$9$}] at (1,2) {};
	\node   (10) [label={below:$10$}] at (2,2) {};
	\node   (11) [label={above:$11$}] at (3,2) {};
	\node   (12) [label={below:$12$}] at (4,2) {};
	\node   (13) [label={above:$13$}] at (5,2) {};
	\node   (14) [label={below:$14$}] at (6,2) {};
	
	\node   (15) [label={below:$15$}] at (1,1) {};
	\node   (16) [label={below:$16$}] at (2,1) {};
	\node   (17) [label={below:$17$}] at (3,1) {};
	\node   (18) [label={below:$18$}] at (4,1) {};
	\node   (19) [label={below:$19$}] at (5,1) {};

	\draw (1)--(2)--(3)--(4)--(5)--(6)--(7);
	\draw (1)--(8)--(9)--(15)--(16)--(17)--(18)--(19)--(13)--(14)--(7);
	\draw (3)--(10)--(9);\draw (5)--(12)--(11)--(10);\draw (11)--(17);
	\draw (5)--(6);\draw (12)--(13);

	\node [draw=none,fill=none] at (0.5,2.5) {\textcolor{blue}{$H_1$}};
	\node [draw=none,fill=none] at (2.5,2.5) {\textcolor{blue}{$H_2$}};
	\node [draw=none,fill=none] at (4.5,2.5) {\textcolor{blue}{$H_3$}};
	\node [draw=none,fill=none] at (1.5,1.5) {\textcolor{blue}{$H_4$}};
	\node [draw=none,fill=none] at (3.5,1.5) {\textcolor{blue}{$H_5$}};

\begin{scope}[shift={(8,0)}]

	\node   (1) [label={above:$1$}] at (0,3) {};
	\node   (2) [label={above:$2$}] at (1,3) {};
	\node   (3) [label={above:$3$}] at (2,3) {};
	\node   (4) [label={above:$4$}] at (3,3) {};
	\node   (5) [label={above:$5$}] at (4,3) {};
	\node   (6) [label={above:$6$}] at (5,3) {};
	\node   (7) [label={above:$7$}] at (6,3) {};
	
	\node   (8) [label={below:$8$}] at (0,2) {};
	\node   (9) [label={above:$9$}] at (1,2) {};
	
	\node   (13) [label={above:$13$}] at (5,2) {};
	\node   (14) [label={below:$14$}] at (6,2) {};
	
	\node   (15) [label={below:$15$}] at (1,1) {};
	\node   (16) [label={below:$16$}] at (2,1) {};
	\node   (17) [label={below:$17$}] at (3,1) {};
	\node   (18) [label={below:$18$}] at (4,1) {};
	\node   (19) [label={below:$19$}] at (5,1) {};

	\draw (1)--(2)--(3)--(4)--(5)--(6)--(7)--(14)--(13)--(19)--(18)--(17)--(16)--(15)--(9)--(8)--(1);

	\node [draw=none,fill=none] at (3,2) {\textcolor{blue}{$H_6$}};	
	
\end{scope}	

\begin{scope}[shift={(0,-8)}]

	\node   (1) [label={above:$1$}] at (0,7) {};
	\node   (2) [label={above:$2$}] at (1,7) {};
	\node   (3) [label={above:$3$}] at (2,7) {};
	\node   (4) [label={above:$4$}] at (3,7) {};
	\node   (5) [label={above:$5$}] at (4,7) {};
	
	\node   (8) [label={below:$8$}] at (0,6) {};
	\node   (9) [label={above:$9$}] at (1,6) {};
	
	\node   (12) [label={below:$12$}] at (4,6) {};
	\node   (13) [label={above:$13$}] at (5,6) {};
	
	\node   (15) [label={below:$15$}] at (1,5) {};
	\node   (16) [label={below:$16$}] at (2,5) {};
	\node   (17) [label={below:$17$}] at (3,5) {};
	\node   (18) [label={below:$18$}] at (4,5) {};
	\node   (19) [label={below:$19$}] at (5,5) {};
	
\draw (1)--(2)--(3)--(4)--(5)--(12)--(13)--(19)--(18)--(17)--(16)--(15)--(9)--(8)--(1);
	\node [draw=none,fill=none] at (2.5,6) {\textcolor{blue}{$H_7$}};

	\node   (23) [label={above:$3$}] at (1,3) {};
	\node   (24) [label={above:$4$}] at (2,3) {};
	\node   (25) [label={above:$5$}] at (3,3) {};
	\node   (26) [label={above:$6$}] at (4,3) {};
	\node   (27) [label={above:$7$}] at (5,3) {};
	
	\node   (29) [label={above:$9$}] at (0,2) {};
	\node   (30) [label={below:$10$}] at (1,2) {};
	\node   (33) [label={above:$13$}] at (4,2) {};
	\node   (34) [label={below:$14$}] at (5,2) {};
	
	\node   (35) [label={below:$15$}] at (0,1) {};
	\node   (36) [label={below:$16$}] at (1,1) {};
	\node   (37) [label={below:$17$}] at (2,1) {};
	\node   (38) [label={below:$18$}] at (3,1) {};
	\node   (39) [label={below:$19$}] at (4,1) {};

	\draw (23)--(24)--(25)--(26)--(27)--(34)--(33)--(39)--(38)--(37)--(36)--(35)--(29)--(30)--(23);
	\node [draw=none,fill=none] at (2.5,2) {\textcolor{blue}{$H_8$}};

\end{scope}	

\begin{scope}[shift={(7,-8)}]

	\node   (1) [label={above:$1$}] at (0,7) {};
	\node   (2) [label={above:$2$}] at (1,7) {};
	\node   (3) [label={above:$3$}] at (2,7) {};
	\node   (4) [label={above:$4$}] at (3,7) {};
	\node   (5) [label={above:$5$}] at (4,7) {};
	
	\node   (8) [label={below:$8$}] at (0,6) {};
	\node   (9) [label={above:$9$}] at (1,6) {};
	\node   (11) [label={above:$11$}] at (3,6) {};
	\node   (12) [label={below:$12$}] at (4,6) {};
		
	\node   (15) [label={below:$15$}] at (1,5) {};
	\node   (16) [label={below:$16$}] at (2,5) {};
	\node   (17) [label={below:$17$}] at (3,5) {};
	
	\draw (1)--(2)--(3)--(4)--(5)--(12)--(11)--(17)--(16)--(15)--(9)--(8)--(1);
	\node [draw=none,fill=none] at (2,6) {\textcolor{blue}{$H_9$}};

	\node   (23) [label={above:$3$}] at (5,5) {};
	\node   (24) [label={above:$4$}] at (6,5) {};
	\node   (25) [label={above:$5$}] at (7,5) {};
	\node   (26) [label={above:$6$}] at (8,5) {};
	\node   (27) [label={above:$7$}] at (9,5) {};
	
	\node   (30) [label={below:$10$}] at (5,4) {};
	\node   (31) [label={above:$11$}] at (6,4) {};
	\node   (33) [label={above:$13$}] at (8,4) {};
	\node   (34) [label={below:$14$}] at (9,4) {};
	
	\node   (37) [label={below:$17$}] at (6,3) {};
	\node   (38) [label={below:$18$}] at (7,3) {};
	\node   (39) [label={below:$19$}] at (8,3) {};

	\draw (23)--(24)--(25)--(26)--(27)--(34)--(33)--(39)--(38)--(37)--(31)--(30)--(23);
	\node [draw=none,fill=none] at (7,4) {\textcolor{blue}{$H_{10}$}};

	\node   (43) [label={above:$3$}] at (1,3) {};
	\node   (44) [label={above:$4$}] at (2,3) {};
	\node   (45) [label={above:$5$}] at (3,3) {};
	
	\node   (49) [label={above:$9$}] at (0,2) {};
	\node   (50) [label={below:$10$}] at (1,2) {};
	\node   (52) [label={below:$12$}] at (3,2) {};
	\node   (53) [label={above:$13$}] at (4,2) {};
	
	\node   (55) [label={below:$15$}] at (0,1) {};
	\node   (56) [label={below:$16$}] at (1,1) {};
	\node   (57) [label={below:$17$}] at (2,1) {};
	\node   (58) [label={below:$18$}] at (3,1) {};
	\node   (59) [label={below:$19$}] at (4,1) {};

	\draw (43)--(44)--(45)--(52)--(53)--(59)--(58)--(57)--(56)--(55)--(49)--(50)--(43);
	\node [draw=none,fill=none] at (2,2) {\textcolor{blue}{$H_{11}$}};

\end{scope}

\end{tikzpicture}\caption{A $5$-wall (top left) together with its eleven holes $H_1, \dots, H_{11}$}
	\label{fig:wall3}

\end{figure}
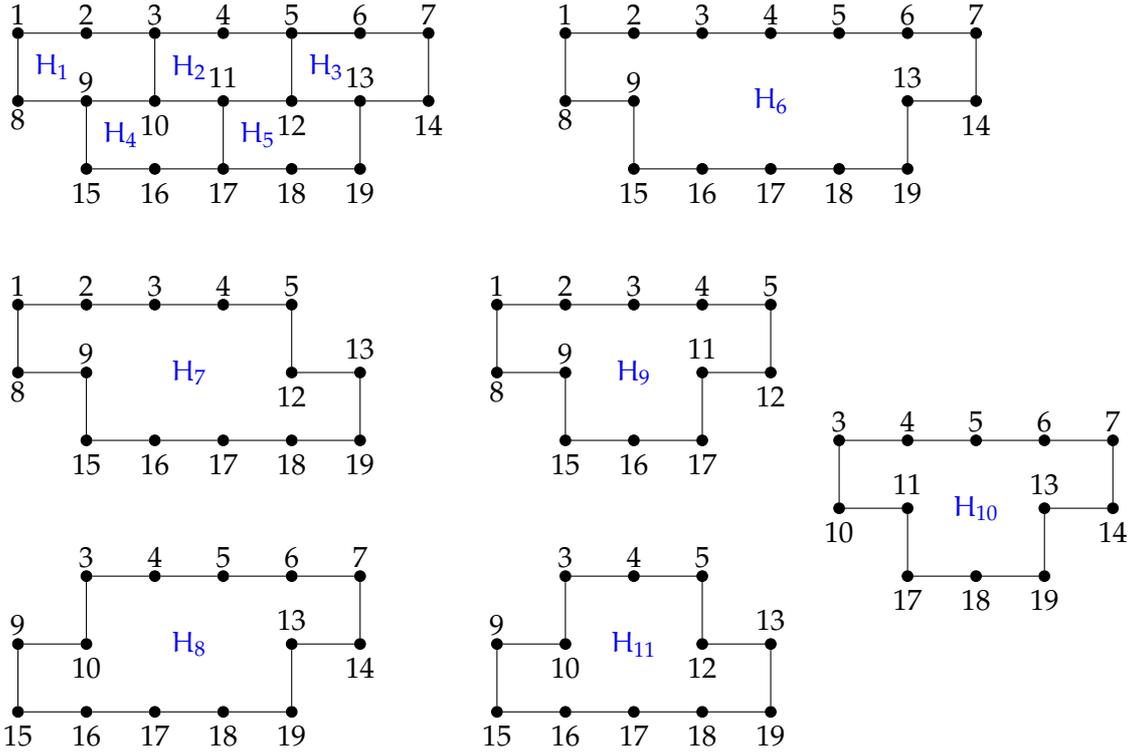

\begin{lemma}\label{lem_Twall}
If $\alpha$ is a nested word  then every subdivision of a $5$-wall in $\RRR^{\alpha}$ contains a base-path-vertex.
\end{lemma}
\begin{proof}
Suppose, for a contradiction, that $\RRR^{\alpha}$ contains a graph $G$ that is a subdivision of a $5$-wall that does not contain a base-path vertex. From Lemma \ref{lem_twostar} there must be precisely two star-vertices in each vertex-set $V_1, \dots, V_{11}$ in $G$.

Firstly, we consider $V_1$ and $V_3$. Notice that if there are two star-vertices in $V_1 \cap V_{8} = V_1 \cap V_{11}$ then there can be no star-vertices in $V_3 \subset (V_8 \cup V_{11})$, a contradiction. So there must be at least one star vertex in $V_1 \setminus V_{11} \subset V_6$ (set $W_1$) and by symmetry at least one star-vertex in  $V_3 \setminus V_{11} \subset V_6$ (set $W_2$). Since these are disjoint subsets of $V_6$ there are no other star-vertices in $V_6$.

Observe that a degree three vertex in a path-star graph must either be a star-vertex or adjacent to a star-vertex. In particular, vertex $17$ in Figure \ref{fig:wall3} is degree three, and combined with the fact that the vertices adjacent to it in $V_6$ are not star-vertices, as demonstrated above, means that there is a star-vertex in $(V_{9}\cap V_{10}) \setminus V_6$ (set $W_3$, that is, the path segment $11$ -- $17$).
  
One star-vertex is in $W_1 \subset V_9$, so another star-vertex is in $V_9\setminus W_1$. Likewise, there is a star-vertex in $W_2 \subset V_{10}$, so another star-vertex is in $V_{10}\setminus W_2$. Both of these vertex sets include $W_3$, which from the previous paragraph must contain a star-vertex.  Consequently, neither $V_9 \setminus (W_1 \cup W_3)$ nor $V_{10} \setminus (W_2 \cup W_3)$ contains a star-vertex.

Combining these sets, $(V_9 \cup V_{10}) \setminus (W_1 \cup W_2 \cup W_3)$ does not contain a star-vertex. But this contains all of $V_2$ except for vertex $11$, so $V_2$ contains at most one star-vertex, a contradiction.
 
Therefore, it is not possible to construct $G$ without a base-path vertex.
\end{proof}

\begin{thm} \label{KKW_free}
If $\alpha$ is a nested word then $\RRR^{\alpha}$ is $KKW$-free.
\end{thm}
\begin{proof}
$\RRR^{\alpha}$ has arboricity two so does not contain $K_5$ or $K_{4,4}$.

We show that $\RRR^{\alpha}$ does not contain a subdivision of a $t \times t$ wall for $t$ where $\alpha$ is $\mathfrak{b}$-nested and $\Big\lfloor\frac{t}{10}\Big\rfloor > \sqrt{\frac{\mathfrak{b}}{3}}$.

Suppose, for a contradiction, that $\RRR^{\alpha}$ contains a graph $G$ that is a subdivision of a $t \times t$ wall for some $t$ where $\Big\lfloor\frac{t}{10}\Big\rfloor > \sqrt{\frac{\mathfrak{b}}{3}}$. Fix some embedding of this wall into $R^{\alpha}$, and let $\SSS \subseteq \AAA$ denote the letters whose star-vertices appear in this embedding. 

Since $\alpha$ is $\mathfrak{b}$-nested, $\SSS$ has a base $\BBB$ where $|\BBB| \le \mathfrak{b}$. For any $x \in \SSS$ there can be at most three $x$-path-vertices in $G$  since  $s_x$ is a vertex of degree at most three. Hence, $V(G)$ can contain at most $3\mathfrak{b} < 9\Big\lfloor\frac{t}{10}\Big\rfloor^2$ base-path-vertices.

From Lemma \ref{lem_Twall} every induced subdivision of a $5$-wall in $G$ must contain a base-path-vertex. Using Figure \ref{fig:5_wall} it is possible to pack at least nine vertex-disjoint $5$-walls into a $10 \times 10$-wall, so our subdivision of a $t \times t$ wall must contain at least $9\Big\lfloor\frac{t}{10}\Big\rfloor^2$ disjoint subdivisions of a $5$-wall.

Hence, allowing at least one base-path-vertex in each induced subdivision of a $5$-wall, $V(G)$ must contain at least  $9\Big\lfloor\frac{t}{10}\Big\rfloor^2$ base-path-vertices. But we know $V(G)$ contains at most $3\mathfrak{b} < 9\Big\lfloor\frac{t}{10}\Big\rfloor^2$ base-path-vertices, so we have a contradiction.  Thus, $\RRR^{\alpha}$ cannot contain a subdivision of a $t \times t$ wall when $\Big\lfloor\frac{t}{10}\Big\rfloor > \sqrt{\frac{\mathfrak{b}}{3}}$.

A similar argument can be applied to a line graph of a subdivision of a $t \times t$ wall noting that each triangle in the line graph contains a star-vertex.
\end{proof}

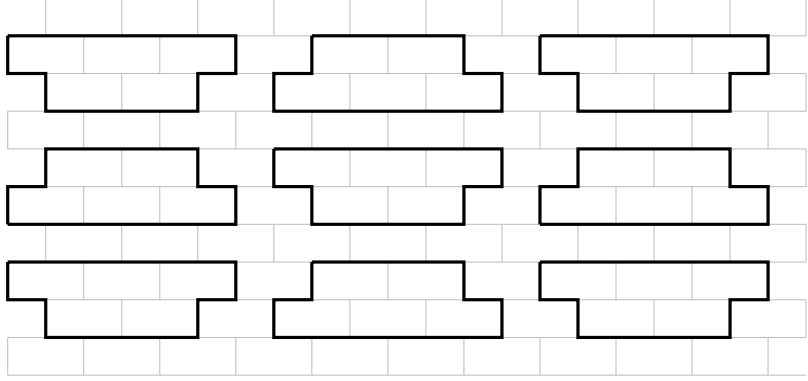
\begin{figure}\centering

\begin{tikzpicture}[scale=0.5,
	vertex2/.style={circle,draw,minimum size=6,fill=black},]

	\draw[lightgray] (1,0)--(22,0);
	\draw[lightgray] (1,1)--(22,1);
	\draw[lightgray] (1,2)--(22,2);
	\draw[lightgray] (1,3)--(22,3);
	\draw[lightgray] (1,4)--(22,4);
	\draw[lightgray] (1,5)--(22,5);
	\draw[lightgray] (1,6)--(22,6);.
	\draw[lightgray] (1,7)--(22,7);
	\draw[lightgray] (1,8)--(22,8);	
	\draw[lightgray] (1,9)--(22,9);
	\draw[lightgray] (1,10)--(22,10);

	\foreach \i in {1,3,5,7,9,11,13,15,17,19,21}
		\foreach \j in {0,2,4,6,8}
	\draw[lightgray] (\i,\j)--(\i,\j+1);
	
	\foreach \i in {2,4,6,8,10,12,14,16,18,20,22}
		\foreach \j in {1,3,5,7,9}
	\draw[lightgray] (\i,\j)--(\i,\j+1);

	\draw[very thick] (1,9)--(1,8)--(2,8)--(2,7)--(6,7)--(6,8)--(7,8)--(7,9)--(1,9);
	\draw[very thick] (9,9)--(9,8)--(8,8)--(8,7)--(14,7)--(14,8)--(13,8)--(13,9)--(9,9);
	\draw[very thick] (15,9)--(15,8)--(16,8)--(16,7)--(20,7)--(20,8)--(21,8)--(21,9)--(15,9);
	
	\draw[very thick] (1,4)--(1,5)--(2,5)--(2,6)--(6,6)--(6,5)--(7,5)--(7,4)--(1,4);
	\draw[very thick] (8,6)--(8,5)--(9,5)--(9,4)--(13,4)--(13,5)--(14,5)--(14,6)--(8,6);
	\draw[very thick] (15,4)--(15,5)--(16,5)--(16,6)--(20,6)--(20,5)--(21,5)--(21,4)--(15,4);

	\draw[very thick] (1,3)--(1,2)--(2,2)--(2,1)--(6,1)--(6,2)--(7,2)--(7,3)--(1,3);
	\draw[very thick] (9,3)--(9,2)--(8,2)--(8,1)--(14,1)--(14,2)--(13,2)--(13,3)--(9,3);
	\draw[very thick] (15,3)--(15,2)--(16,2)--(16,1)--(20,1)--(20,2)--(21,2)--(21,3)--(15,3);

\end{tikzpicture}
\caption{Nine $5$-walls in a $10 \times 10$ wall}
	\label{fig:5_wall}

\end{figure}

%
%
%
%
%
%
\subsection{A nested path-star graph with large tree-width contains a  large \texorpdfstring{$t$}{t{}}-sail}\label{Section:largetw}

A \emph{$k$-block} in a graph $G$ is a maximal set of at least $k$ vertices no two of which can be separated in $G$ by deleting fewer than $k$ vertices. A $k$-block can be thought of as a highly connected part of a graph and has been used in a number of ways. In particular, Wei{\ss}auer showed in  \cite{weissauer:blocks:} that for $k \ge 1$ every graph of tree-width at least $2k^2$ has a minor containing a $k$-block.

In \cite{abrishami:induced_subgraphs_VII:} a more restricted type of $k$-block was introduced. A \emph{strong $k$-block} in $G$ is a set $B$ of at least $k$ vertices such that for every $2$-subset $\{x,y\}$ of $B$, there exists a collection $\PPP_{x,y}$ of at least $k$ distinct and pairwise internally disjoint paths in $G$ from $x$ to $y$, where for every two distinct $2$-subsets $\{x,y\},\{x',y'\} \subseteq B$ and every choice of paths $P \in \PPP_{x,y}$ and $P' \in \PPP_{x',y'}$ we have $P \cap P'=\{x,y\} \cap \{x',y'\}$.

We show that all $t$-sails for large $t$ contain strong $k$-blocks for large $k$ and that in nested path-star graph classes, strong $k$-blocks for large $k$ only occur in graphs containing a $t$-sail for large $t$ as an induced subgraph . We use this to conclude that a nested path-star graph has large tree-width if and only if it contains a $t$-sail for large $t$ as an induced subgraph.

\begin{lemma}\label{block_1}
For any $t \ge 1$  a $t^3$-sail contains a strong $t$-block.
\end{lemma}
\begin{proof}
Let $B=\{s_1, \dots s_t\}$, i.e., the first $t$ star-vertices. We claim $B$ is a strong $t$-block.

For every $2$-subset $\{s_x,s_y\}$ of $B$, we define the set $\PPP_{x,y}$ of $t$ disjoint paths between $s_x$ and $s_y$ ($x<y$) being the $t$ class-path components numbered $\{(x-1)t^2+(y-1)t, (x-1)t^2+(y-1)t+1, \dots, (x-1)t^2+(y-1)t+t-1\}$.

For every two distinct $2$-subsets $\{x,y\},\{x',y'\} \subseteq B$ and every choice of paths $P \in \PPP_{x,y}$ and $P' \in \PPP_{x',y'}$ we have $P \cap P'=\{s_x,s_y\} \cap \{s_x',s_y'\}$. Hence $B$ is a strong $t$-block.   
\end{proof}

To show that in nested path-star graph classes a large strong $k$-block contains a large $t$-sail as an induced subgraph, we explore the structure of nested words further.

Let $\SSS_1$ be the infinite set of all letters appearing in the $\mathfrak{b}$-nested word $\alpha$. We write $\BBB_1$ for the base of $\alpha$. Define nested subwords $\alpha^{\SSS_i}$ for $i \ge 2$ by $\SSS_i= \SSS_{i-1} \setminus \BBB_{i-1}$ where $\BBB_i$ is the base of $\alpha^{\SSS_i}$. Note that each $\BBB_i$ contains at least one and at most $\mathfrak{b}$ letters.

For any positive integer $t$, define a \emph{$\BBB^t$-factor} as a sequence of $\alpha$ containing only letters from $\cup_{i=1}^{t} \BBB_i$, and at least one letter from each $\BBB_i$ for $1 \le i \le t$. Likewise, define a \emph{$\SSS^t$-factor} as a sequence of $\alpha$ containing only letters from a single branch of $\alpha^{\SSS_t}$.

\begin{obs}\label{obs_nest_1}
For any positive integer $t$, a nested word $\alpha$ alternates between $\BBB^t$-factors and $\SSS^t$-factors.
\end{obs}
\begin{proof}
In a branched word (see Definition \ref{def_branched}), one or more base letters must appear immediately before and immediately after a branch, so the statement is true for $t=1$.
Using induction, assume the statement is true for $t= n-1$ so that $\alpha$ alternates between $\BBB^{n-1}$-factors and $\SSS^{n-1}$-factors.

Suppose, in $\alpha$, we have a sequence $uvw$ where $u$ and $w$ are $\SSS^{n-1}$-factors and $v$ is a $\BBB^{n-1}$-factor. From Definition \ref{def_nested}, in $\alpha^{\SSS_n}$  each of $u$ and $w$ are reduced to a factor that contains at most one $\SSS^{n}$-factor (since from bullet point three of the definition it is not possible to get two branches in $\alpha^{\SSS_n}$ out of one branch in $\alpha^{\SSS_{n-1}}$) and $v$ completely disappears.

As $\alpha^{\SSS_n}$ is branched we cannot have two  adjacent $\SSS^n$-factors so there must be a letter(s) from $\BBB_n$ between them. Combining this letter(s) with the $\BBB^{n-1}$-factor $v$, gives us a  $\BBB^{n}$-factor between every pair of $\SSS^{n}$-factors, so the statement is true for $t= n$.

Therefore, from the induction hypothesis, the observation follows.
\end{proof}

\begin{obs}\label{obs_nest_2}
Let $\HHH = \{ h_1, h_2, \dots \}$ be a branch set of $\alpha^{\SSS_t}$ where $h_1<h_2< \dots   $. Then for any $x<y$, between any occurrence of $h_x$ and $h_y$ in $\alpha$ there is either a factor $h_1 \dots h_{x-1}$ or $h_{x+1} \dots h_{y-1}$.
\end{obs}
\begin{proof}
This follows from the fact that branches must start with the first letter in the branch set and must appear in branch order. 
\end{proof}

\begin{lemma}\label{block_2}
Let $G$ be a graph from a path-star class defined by a $\mathfrak{b}$-nested word $\alpha$ over the infinite alphabet $\AAA$.  If $G$ contains a strong $k$-block, where $k \ge \max \{t\mathfrak{b}^{t}+t\mathfrak{b},t(\mathfrak{b}+2)+2\}$ for some integer $t \ge 1$, then it also contains a $t$-sail as an induced subgraph.
\end{lemma}
\begin{proof}
Fix some embedding of $G$ into $R^{\alpha}$, and let $\SSS_1 \subseteq \AAA$ denote the letters whose star-vertices appear in this embedding.  For any two vertices in a strong $k$-block there must be $k$ internally disjoint paths between them.   The vertices in this strong $k$-block must be star-vertices since only star-vertices can have degree greater than three, so let $\LLL\subseteq \SSS_1$ be the letters corresponding to the vertices of the strong $k$-block in $G$ where $k \ge t\mathfrak{b}^{t}+t\mathfrak{b}$.

Let subword $\alpha^{\SSS_1}$ have base $\BBB_1$, and define subwords $\alpha^{\SSS_i}$ for $i \ge 2$ by $\SSS_i= \SSS_{i-1} \setminus \BBB_{i-1}$ with base $\BBB_i$. 

Observe that each star-vertex corresponding to a letter in $\SSS_1 \setminus \LLL$ appears in at most one of the internally disjoint paths between two vertices of the strong $k$-block as otherwise the paths would not be disjoint.

As $k \ge t(\mathfrak{b}+2)+2$, and there are at most $t\mathfrak{b}$ letters in $\cup_{i=1}^{t} \BBB_i$,  $\LLL$ contains at least $(2t+2)$ letters in $\SSS_{t+1}$. Therefore, either there exists a pair $x,y \in \LLL$ that are in different branch sets of $\alpha^{\SSS_t}$ (Case $1$) or there are at least $(2t+2)$ letters in the same branch set of $\alpha^{\SSS_t}$ (Case $2$).

Case 1: (There exists a pair $x,y \in \LLL$ in different branch sets of $\alpha^{\SSS_t}$.) From Observation \ref{obs_nest_1}, between every occurrence of $x$ and occurrence of $y$ in $\alpha$ there is a $\BBB^t$ factor.

At most one of the disjoint paths from $s_x$ to $s_y$ can pass through each star-vertex associated with a letter in $\cup_{i=1}^{t} \BBB_i$ (i.e., at most $t\mathfrak{b}$ paths). This leaves $k- t\mathfrak{b}$ paths that do not pass through such a star-vertex. The remaining disjoint paths must all include a set of consecutive class-path-vertices corresponding to a $\BBB^t$ factor in $\alpha$ (Observations \ref{obs_nest_1}).

Given a collection of  $\BBB^t$ factors of size $k- t\mathfrak{b}$, using the pigeonhole principle, as $k \ge t\mathfrak{b}^{t}+t\mathfrak{b}$, there are at least $\frac {k- t\mathfrak{b}}{\mathfrak{b}^{t}} \ge t$ of them that contain the same letter from each set $\BBB_i$, $1 \le i \le t$. Call this set of at least $t$ letters $\TTT \subset \cup_{i=1}^{t} \BBB_i$. It follows that at least $t$ of the disjoint paths from $s_x$ to $s_y$ contain a component of the class-path incorporating a path-vertex corresponding to each letter in $\TTT$ -- let us call these components $I_1, \dots I_t$.

Case 2: (There exists at least $(2t+2)$ letters in the same branch set of $\alpha^{\SSS_t}$.)  
Let the $(2t+2)$ letters come from branch $\HHH = \{ h_1, h_2, \dots \}$ where $h_1<h_2< \dots$. Since  there are at least $(2t+2)$ letters, we can choose the $(t+1)$-th and $(2t+2)$-th letters as $h_x$ and $h_y$ so that $t<x<x+t<y$.
Using Observation \ref{obs_nest_2}, between every occurrence of $h_x$ and occurrence of $h_y$ in $\alpha$ there is a set of $t$ consecutive letters from $\HHH$ (either the first $t$ letters in $\HHH$ or the (at least) $t$ letters between $h_x$ and $h_y$ in the $\HHH$ order).

Denoting $s_x$ and $s_y$ as the two star-vertices in the $k$-block corresponding to letters $h_x$ and $h_y$, at most one of the disjoint paths from $s_x$ to $s_y$ can pass through each star-vertex associated with one of these $2t$ letters in $\HHH$. This leaves $k- 2t$ disjoint paths that do not pass through such a star-vertex. The remaining disjoint paths must all include a set of consecutive class-path-vertices corresponding to one of our two sets of consecutive letters from $\HHH$ (Observation \ref{obs_nest_2}).

From the pigeonhole principle at least $\frac{k-2t}{2} \ge t$ of the disjoint paths must correspond to the same set of $t$ consecutive letters from $\HHH$. Call this set of letters $\TTT \subseteq \HHH$. It follows that at least $t$ of the disjoint paths from $s_x$ to $s_y$ contain a component of the class-path incorporating a path-vertex corresponding to each letter in $\TTT$ -- let us call these components $I_1, \dots I_t$.

In either Case 1 or Case 2,  we have path components in $G$, $I_1, \dots I_t$,  each containing path-vertices corresponding to each letter in $\TTT$. Given that $\TTT \subset \SSS_1$ the stars $S_{\TTT}$ corresponding to letters in $\TTT$ are all in V[G]. Therefore, the graph $G\Big[ \bigcup_{i=1}^t V[I_i] \cup S_{\TTT}\Big]$  contains a forest of $t$ paths and a forest of $t$ stars  sufficient to fulfil the definition that it contains a $t$-sail as an induced subgraph [see example in Figure \ref{fig:k_block}].
\end{proof}

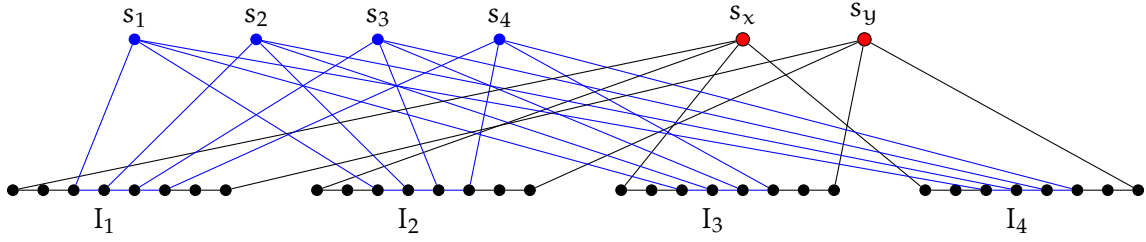
\begin{figure}\centering

\begin{tikzpicture}[scale=0.4,
	vertex2/.style={circle,draw,minimum size=5,fill=red},]
	
	\node (1) at(-10,0) {};\node (2) at(-9,0) {};\node (3) at (-8,0) {};
	\node (4) at(-7,0) {};\node (5) at(-6,0) {};\node (6) at (-5,0) {};
	\node (7) at(-4,0) {};\node (8) at(-3,0) {};
	\draw (1)--(2)--(3);\draw [blue](3)--(4)--(5)--(6);\draw(6)--(7)--(8);
	
	\node (11) at(0,0) {};\node (12) at(1,0) {};\node (13) at (2,0) {};
	\node (14) at(3,0) {};\node (15) at(4,0) {};\node (16) at (5,0) {};
	\node (17) at(6,0) {};\node (18) at(7,0) {};
	\draw (11)--(12)--(13);\draw [blue](13)--(14)--(15)--(16);\draw(16)--(17)--(18);
	
	\node (21) at(10,0) {};\node (22) at(11,0) {};\node (23) at (12,0) {};
	\node (24) at(13,0) {};\node (25) at(14,0) {};\node (26) at (15,0) {};
	\node (27) at(16,0) {};\node (28) at(17,0) {};
	\draw (21)--(22)--(23);\draw[blue](23)--(24)--(25)--(26);\draw(26)--(27)--(28);
	
	\node (31) at(20,0) {};\node (32) at(21,0) {};\node (33) at (22,0) {};
	\node (34) at(23,0) {};\node (35) at(24,0) {};\node (36) at (25,0) {};
	\node (37) at(26,0) {};\node (38) at(27,0) {};
	\draw (31)--(32)--(33);\draw [blue](33)--(34)--(35)--(36);\draw(36)--(37)--(38);
	
	\node (s1) [color=blue,label={above:$s_1$}] at (-6,5) {};
	\node (s2) [color=blue,label={above:$s_2$}] at (-2,5) {};
	\node (s3) [color=blue,label={above:$s_3$}] at (2,5) {};
	\node (s4) [color=blue,label={above:$s_4$}] at (6,5) {};
	
	\node (sx) [vertex2][label={above:$s_x$}] at (14,5) {};
	\node (sy) [vertex2][label={above:$s_y$}] at (18,5) {};

	\draw (sx) -- (11);\draw [blue](s1) -- (13);\draw [blue](s2) -- (14); 
	\draw [blue](s3) -- (15);\draw [blue](s4) -- (16);\draw(sy) -- (18); 
	
	\draw (sx) -- (21);\draw [blue](s1) -- (23);\draw [blue](s2) -- (24); 
	\draw [blue](s3) -- (25);\draw [blue](s4) -- (26);\draw(sy) -- (28); 
	
	\draw (sx) -- (31);\draw [blue](s1) -- (33);\draw [blue](s2) -- (34); 
	\draw [blue](s3) -- (35);\draw [blue](s4) -- (36);\draw(sy) -- (38); 
	
	\draw (sx) -- (1);\draw [blue](s1) -- (3);\draw [blue](s2) -- (4); 
	\draw [blue](s3) -- (5);\draw [blue](s4) -- (6);\draw(sy) -- (8); 
	
	\node (101)[draw=none,fill=none] at (-7,-1) {$I_1$};
	\node (102)[draw=none,fill=none] at (3,-1) {$I_2$};
	\node (103)[draw=none,fill=none] at (13,-1) {$I_3$};
	\node (104)[draw=none,fill=none] at (23,-1) {$I_4$};

\end{tikzpicture}

\caption{A $4$-sail (blue) with $s_x$ and $s_y$ two nodes in a $k$-block}
	\label{fig:k_block}

\end{figure}

Letting $\mathbb{B}_k$ be the class of all graphs with no strong $k$-block and remembering that the \emph{$k$-basic obstructions} are $(1)$ the complete graph $K_k$, $(2)$ the complete bipartite graph $K_{k,k}$, $(3)$ a subdivision of the $(k \times k)$-wall and $(4)$ a line graph of a subdivision of the $(k \times k)$-wall, we use the following result:

\begin{thm} [\cite{abrishami:induced_subgraphs_VII:}]\label{block_3}
For every integer $k \ge 1$ there exists a positive integer $w(k)$ such that every graph in $\mathbb{B}_k$ with tree-width more than $w(k)$ contains an induced subgraph isomorphic to one of the $k$-basic obstructions.
\end{thm}

\begin{thm}\label{thm_sail}
If $\alpha$ is a nested word then for every $t \ge 1$ there is a positive integer valued function $f_{\alpha}(t)$ such that  every graph in $\RRR^{\alpha}$ of tree-width at least $f_{\alpha}(t)$ contains a $t$-sail as an induced subgraph.   
\end{thm}
\begin{proof}
Let $k = \max \{t\mathfrak{b}^{t}+t\mathfrak{b},t(\mathfrak{b}+2)+2\}$ and $f_{\alpha}(t)=w(k)$ as defined by Theorem \ref{block_3}. Suppose for graph $G \in \RRR^{\alpha}$, we have $\tw(G) \ge f_{\alpha}(t)$. Then by Theorem \ref{block_3}, $G$ cannot be in $\mathbb{B}_k$ because by Theorem \ref{KKW_free} $\RRR^{\alpha}$ is $KKW$-free, $G$ does not contain a $k$-basic obstruction, and therefore, $G$ contains a strong $k$-block. It follows by Lemma \ref{block_2} that $G$ contains an induced subgraph isomorphic to a $t$-sail. 
\end{proof} 

%
%
%
%
%
%
\subsection{Nested path-star classes are infinitely defined}\label{Sect:infinit_defn}

As previously mentioned, the list of minimal forbidden induced subgraphs in a hereditary class may be finite or infinite. The following result characterizes hereditary classes of unbounded tree-width that are finitely defined:
 
\begin{thm}[\cite{lozin:tw_dichotomy:}]\label{thm_dichotomy}
The tree-width of graphs in a hereditary class defined by a finite set $\FFF$  of forbidden induced subgraphs is bounded if and only if $\FFF$ includes a complete graph, a complete bipartite graph, a tripod (a forest in which every connected component has at most $3$ leaves) and the line graph of a tripod.
\end{thm}

A full classification of hereditary classes into finitely or infinitely defined is a long way from being established. However, we can show that one consequence of Theorem \ref{thm_dichotomy} is the following:

\begin{thm} \label{thm_infinite_def_1}
A hereditary class of graphs of unbounded tree-width that is $KKW$-free  is infinitely defined.
\end{thm}
\begin{proof}
Let $\CCC$  be a hereditary class of graphs of unbounded tree-width that is $KKW$-free, so that it excludes a subdivision of a $t \times t$-wall  and line graph of a subdivision of a $t \times t$-wall  for some $t \in \mathbb{N}$. 

For a contradiction suppose $\CCC$ is finitely defined with minimal forbidden induced subgraphs $\FFF=\{F_1, \dots, F_n\}$ for some $n \in \mathbb{N}$. As $\CCC$ has unbounded tree-width, then from Theorem \ref{thm_dichotomy} $\FFF$ either does not contain a tripod or does not contain the line graph of a tripod.
Suppose it does not contain a tripod (i.e. all tripods are in $\CCC$).

Let $d$ denote the maximum distance between any two vertices of degree three or more for any tree (other than a tripod) in $\FFF$, and let $m$ denote the maximum length of any induced cycle in any graph in $\FFF$.

Let $W$ be a subdivision of a $t \times t$-wall with more than $\max(d,m)$ degree two vertices on the paths between degree three vertices. This contains none of the minimal forbidden induced subgraphs in $\FFF$ since every induced cycle of $W$ has length more than $m$ and every induced tree of $W$ has distance greater than $d$ between degree three vertices, hence, $W \in \CCC$.  A contradiction of the assumption that $\CCC$ excludes a subdivision of a $t \times t$-wall.

A similar argument applies if $\FFF$ does not contain the line graph of a tripod, thus $\CCC$ must be infinitely defined.
\end{proof}
\begin{cor}
A path-star class defined by a nested word is infinitely defined. 
\end{cor}
\begin{proof}
From Theorem \ref{KKW_free} all path-star classes defined by a nested word are $KKW$-free  so from Theorem \ref{thm_infinite_def_1} are infinitely defined.
\end{proof}

%
%
%
%
%
%
\section{Minimal sparse hereditary classes of unbounded tree-width}\label{Sect_min}

%
%
%
%
%
%
\subsection{Hereditary graph classes of bounded vertex degree or with an excluded minor do not contain a minimal subclass} \label{Sect:Wall}

The structure of a wall allows us to delete vertices and leave the fundamental structure intact, ignoring subdivisions, and this quality is used in the following:

\begin{lemma}\label{lem_subwall}
A subdivision (or line graph of a subdivision) of a $W_{kt \times kt}$ wall  for $k,t \ge 1$ contains an induced subgraph isomorphic to a subdivision (or line graph of a subdivision) of a $W_{t \times t}$ wall  that does not contain a cycle smaller than $C_{8k-2}$ (other than $C_3$ in the case of the line graph). 
\end{lemma}
\begin{proof} 
Let $G=(V,E)$ be a subdivision of a $W_{kt \times kt}$ wall. Let $V^3 \subseteq V$ be the set of degree three vertices in $G$, together with the equivalent degree two vertices from the perimeter 'bricks' (or holes) that would be degree three if the wall was extended, so that every brick in $G$  contains six vertices in $V^3$. An induced subgraph $G^{\prime}$ isomorphic to  a subdivision of a $W_{t \times t}$ wall can be constructed by overlaying a lattice of $k \times k$ sub-walls, the new 'bricks', and deleting all vertices of $G$ internal to every new brick, as shown in the example in Figure \ref{fig:wall} where $k=2$ and $t=4$. Each new brick contains $8k-2$ vertices from $V^3$, and thus $G^{\prime}$ does not contain a cycle smaller than $C_{8k-2}$.

An identical argument works if $G$ is the line graph of a subdivision of  a $W_{kt \times kt}$ wall.
\end{proof}

\begin{thm} \label{min_w}
If $\CCC$ is a hereditary class of graphs of bounded vertex degree or that has an excluded minor then it does not contain a minimal class.
\end{thm}
\begin{proof}
If $\DDD$ is a minimal hereditary subclass of $\CCC$ then  by Theorems \ref{thm_minor_free} or \ref{thm_degree_bounded}, as it has unbounded tree-width, it contains (as a member of the class) a graph $G$ which is isomorphic to a subdivision (or line graph of a subdivision) of a $W_{kt \times kt}$ wall for arbitrarily large $k$ and $t$.

Suppose $C_m$ ($m>3$) is the shortest cycle in $G$. Set $k > \frac{m+2}{8}$.  Then from Lemma \ref{lem_subwall} $G$ contains as an induced subgraph a subdivision (or line graph of a subdivision) of a $W_{t \times t}$ wall that does not contain a cycle smaller than $C_{8k-2}$ which is longer than $C_m$ (other than $C_3$).

But now the proper hereditary subclass $\DDD \cap Free(C_m)$ contains a subdivision of $W_{t \times t}$ for arbitrarily large $t$, so 
$\DDD \cap Free(C_m)$ also has unbounded tree-width, which  contradicts $\DDD$ being minimal.
\end{proof}

\begin{figure}\centering

\begin{tikzpicture}[scale=0.5,
	vertex2/.style={circle,draw,minimum size=6,fill=black},]

	\draw[lightgray] (2,8)--(3,8)--(4,8)--(5,8);
	\draw[lightgray] (3,9)--(3,8)--(4,8)--(4,7); 
	\draw[lightgray] (6,8)--(7,8)--(8,8)--(9,8);
	\draw[lightgray] (7,9)--(7,8)--(8,8)--(8,7);
	\draw[lightgray] (10,8)--(11,8)--(12,8)--(13,8);
	\draw[lightgray] (11,9)--(11,8)--(12,8)--(12,7);
	\draw[lightgray] (14,8)--(15,8)--(16,8)--(17,8);
	\draw[lightgray] (15,9)--(15,8)--(16,8)--(16,7);
	
	\draw[lightgray] (2,6)--(3,6)--(4,6)--(5,6);
	\draw[lightgray] (3,7)--(3,6)--(4,6)--(4,5); 
	\draw[lightgray] (6,6)--(7,6)--(8,6)--(9,6);
	\draw[lightgray] (7,7)--(7,6)--(8,6)--(8,5);
	\draw[lightgray] (10,6)--(11,6)--(12,6)--(13,6);
	\draw[lightgray] (11,7)--(11,6)--(12,6)--(12,5);
	\draw[lightgray] (14,6)--(15,6)--(16,6)--(17,6);
	\draw[lightgray] (15,7)--(15,6)--(16,6)--(16,5);
	
	\draw[lightgray] (2,4)--(3,4)--(4,4)--(5,4);
	\draw[lightgray] (3,5)--(3,4)--(4,4)--(4,3); 
	\draw[lightgray] (6,4)--(7,4)--(8,4)--(9,4);
	\draw[lightgray] (7,5)--(7,4)--(8,4)--(8,3);
	\draw[lightgray] (10,4)--(11,4)--(12,4)--(13,4);
	\draw[lightgray] (11,5)--(11,4)--(12,4)--(12,3);
	\draw[lightgray] (14,4)--(15,4)--(16,4)--(17,4);
	\draw[lightgray] (15,5)--(15,4)--(16,4)--(16,3);
	
	\draw[lightgray] (2,2)--(3,2)--(4,2)--(5,2);
	\draw[lightgray] (3,3)--(3,2)--(4,2)--(4,1); 
	\draw[lightgray] (6,2)--(7,2)--(8,2)--(9,2);
	\draw[lightgray] (7,3)--(7,2)--(8,2)--(8,1);
	\draw[lightgray] (10,2)--(11,2)--(12,2)--(13,2);
	\draw[lightgray] (11,3)--(11,2)--(12,2)--(12,1);
	\draw[lightgray] (14,2)--(15,2)--(16,2)--(17,2);
	\draw[lightgray] (15,3)--(15,2)--(16,2)--(16,1);

	\draw[very thick] (1,9)--(1,8)--(2,8)--(2,7)--(3,7)--
	(4,7)--(5,7)--(6,7)--(6,8)--(5,8)--(5,9)--(4,9)--
	(3,9)--(2,9)--(1,9);
	\draw[very thick] (5,9)--(5,8)--(6,8)--(6,7)--(7,7)--
	(8,7)--(9,7)--(10,7)--(10,8)--(9,8)--(9,9)--(8,9)--
	(7,9)--(6,9)--(5,9);
	\draw[very thick] (9,9)--(9,8)--(10,8)--(10,7)--
	(11,7)--(12,7)--(13,7)--(14,7)--(14,8)--(13,8)--
	(13,9)--(12,9)--(11,9)--(10,9)--(9,9);
	\draw[very thick] (13,9)--(13,8)--(14,8)--(14,7)--
	(15,7)--(16,7)--(17,7)--(18,7)--(18,8)--(17,8)--
	(17,9)--(16,9)--(15,9)--(14,9)--(13,9);
	
	\draw[very thick] (1,7)--(1,6)--(2,6)--(2,5)--(3,5)--
	(4,5)--(5,5)--(6,5)--(6,6)--(5,6)--(5,7)--(4,7)--
	(3,7)--(2,7)--(1,7);
	\draw[very thick] (5,7)--(5,6)--(6,6)--(6,5)--(7,5)--
	(8,5)--(9,5)--(10,5)--(10,6)--(9,6)--(9,7)--(8,7)--
	(7,7)--(6,7)--(5,7);
	\draw[very thick] (9,7)--(9,6)--(10,6)--(10,5)--
	(11,5)--(12,5)--(13,5)--(14,5)--(14,6)--(13,6)--
	(13,7)--(12,7)--(11,7)--(10,7)--(9,7);
	\draw[very thick] (13,7)--(13,6)--(14,6)--(14,5)--
	(15,5)--(16,5)--(17,5)--(18,5)--(18,6)--(17,6)--
	(17,7)--(16,7)--(15,7)--(14,7)--(13,7);
	
	\draw[very thick] (1,5)--(1,4)--(2,4)--(2,3)--(3,3)--
	(4,3)--(5,3)--(6,3)--(6,4)--(5,4)--(5,5)--(4,5)--
	(3,5)--(2,5)--(1,5);
	\draw[very thick] (5,5)--(5,4)--(6,4)--(6,3)--(7,3)--
	(8,3)--(9,3)--(10,3)--(10,4)--(9,4)--(9,5)--(8,5)--
	(7,5)--(6,5)--(5,5);
	\draw[very thick] (9,5)--(9,4)--(10,4)--(10,3)--
	(11,3)--(12,3)--(13,3)--(14,3)--(14,4)--(13,4)--
	(13,5)--(12,5)--(11,5)--(10,5)--(9,5);
	\draw[very thick] (13,5)--(13,4)--(14,4)--(14,3)--
	(15,3)--(16,3)--(17,3)--(18,3)--(18,4)--(17,4)--
	(17,5)--(16,5)--(15,5)--(14,5)--(13,5);
	
	\draw[very thick] (1,3)--(1,2)--(2,2)--(2,1)--(3,1)--
	(4,1)--(5,1)--(6,1)--(6,2)--(5,2)--(5,3)--(4,3)--
	(3,3)--(2,3)--(1,3);
	\draw[very thick] (5,3)--(5,2)--(6,2)--(6,1)--(7,1)--
	(8,1)--(9,1)--(10,1)--(10,2)--(9,2)--(9,3)--(8,3)--
	(7,3)--(6,3)--(5,3);
	\draw[very thick] (9,3)--(9,2)--(10,2)--(10,1)--
	(11,1)--(12,1)--(13,1)--(14,1)--(14,2)--(13,2)--
	(13,3)--(12,3)--(11,3)--(10,3)--(9,3);
	\draw[very thick] (13,3)--(13,2)--(14,2)--(14,1)--
	(15,1)--(16,1)--(17,1)--(18,1)--(18,2)--(17,2)--
	(17,3)--(16,3)--(15,3)--(14,3)--(13,3);
	
	\mnnodearray{18}{9}
	
	\foreach \x/\y in 
	{3/8,4/8,7/8,8/8,11/8,12/8,15/8,16/8,
	3/6,4/6,7/6,8/6,11/6,12/6,15/6,16/6,
	3/4,4/4,7/4,8/4,11/4,12/4,15/4,16/4,
	3/2,4/2,7/2,8/2,11/2,12/2,15/2,16/2}
		\node[vertex2,green!70!black] (\x-\y) at (\x,\y) {};

\end{tikzpicture}
\caption{An $8 \times 8$ wall containing a subdivision of a $4 \times 4$ wall after large (green) vertex deletion}
	\label{fig:wall}

\end{figure}
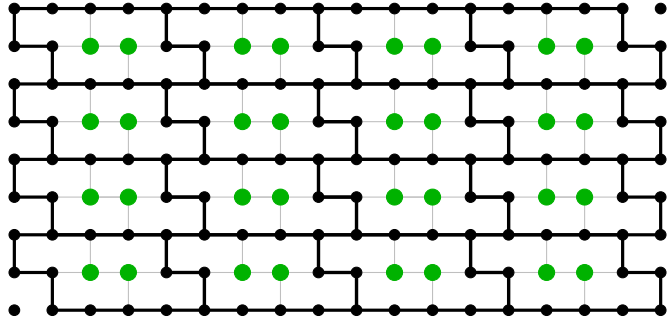

%
%
%
%
%
%
\subsection{Nested path-star hereditary graph classes do not contain a minimal subclass}

We will show that no path-star hereditary graph class defined by a  nested word contains a minimal subclass. 

\begin{lemma} \label{min_sail}
Let $\alpha$ be a $\mathfrak{b}$-nested word over an infinite alphabet $\AAA$. Then for integers $t \ge 2$, $m > 3$ and $T \ge 2t+m\mathfrak{b}$, a $T$-sail in $\RRR^{\alpha}$ with smallest cycle $C_m$ contains an induced subgraph isomorphic to a $t$-sail  which does not contain  $C_m$.
\end{lemma}
\begin{proof}
Let $G$ be a $T$-sail in $\RRR^{\alpha}$. Fix some embedding of $G$ into $R^{\alpha}$, and let $\SSS_1 \subseteq \AAA$ denote the letters whose star-vertices appear in this embedding.  

Let subword $\alpha^{\SSS_1}$ have base $\BBB_1$, and define nested subwords $\alpha^{\SSS_i}$ for $2 \le i \le m$ by $\SSS_i= \SSS_{i-1} \setminus \BBB_{i-1}$ with base $\BBB_i$ (Note $|\BBB_i| \le \mathfrak{b}$ for all $i= 1,2, \dots$).

Let $G^{\prime}$ be the subgraph of $G$ induced by the vertices of $G$ excluding the star-vertices corresponding to the letters of $\cup_{i=1}^m \BBB_i$ (at most $m\mathfrak{b}$) and excluding the star-vertices corresponding to alternate letters in the branch sets of $\SSS_m$ (i.e., half the remaining star-vertices).

We claim $G^{\prime}$ contains an induced subgraph isomorphic to a $t$-sail  which does not contain  $C_m$.

That $G^{\prime}$ contains an induced subgraph isomorphic to a $t$-sail follows from the fact that it contains at least $\frac{T-m\mathfrak{b}}{2} \ge t$ of the star-vertices of $G$ and all the path-vertices from the path components of $G$.

Suppose that $G^{\prime}$ contains an $m$-cycle with only one star-vertex, say $s_x$. Then the path-vertices adjacent to $s_x$ in the cycle must both correspond to a branch letter $x$ of $\SSS_m$. The rest of the cycle must consist of path-vertices corresponding to a factor $x \dots x$ of $\alpha$. Using Observations \ref{obs_nest_1} and \ref{obs_nest_2} there must be at least one base path-vertex from each of the $m$ base sets $\BBB_i$ in the cycle so it has more than $m$ vertices, a  contradiction.

Suppose that $G^{\prime}$ contains an $m$-cycle with three (or more)  star-vertices, $s_x$, $s_y$ and $s_z$. This must contain path components corresponding to $\alpha$ factors of the form $x \dots y$, $x\dots z$ and $y \dots z$ (or their reverse). These factors must be contained in branches of $\alpha^{\SSS_m}$, and hence $x$, $y$ and $z$ must be from the same branch set, otherwise the cycle would have more than $m$ vertices for the same reason as for the one star-vertex case.
But now, assuming without loss of generality that their branch order is $x<y<z$, then there must be a $y$ in the $x\dots z$ factor and a shorter cycle exists, which contradicts the fact that the shortest cycle in $G$ is of length $m$.

 Lastly, suppose that $G^{\prime}$ contains an $m$-cycle with precisely two  star-vertices, $s_x$ and $s_z$. This must contain path components corresponding to two $\alpha$ factors of the form $x \dots z$ (or the reverse). Hence, $x$ and $z$ are from the same branch set. In fact they must be consecutive letters from the same branch set, since if there was another letter $y$ between them in the branch order then there would be a shorter cycle than $C_m$, either containing the two stars $s_x$ and $s_y$ or the two stars $s_y$ and $s_z$. But the construction of $G^{\prime}$ requires the removal of star-vertices corresponding to alternate letters in the branch sets of $\SSS_m$, so $x$ and $z$ cannot be consecutive branch letters, a contradiction. Hence $G^{\prime}$ does not contain  $C_m$.
\end{proof}

\begin{thm} \label{min_nested_2}
If $\RRR^{\alpha}$ is a path-star hereditary class of graphs defined by a nested word $\alpha$ then it does not contain a minimal class.
\end{thm}
\begin{proof}
If $\DDD$ is a minimal subclass of $\RRR^{\alpha}$ then  by Theorem \ref{thm_sail}  for every positive integer $T$, $\DDD$ contains a $T$-sail.

Suppose the shortest cycle in $\DDD$ is $C_m$ ($m > 3$). Then from Lemma \ref{min_sail}  for any positive integer $t$ there exists a positive integer $T$ such that any $T$-sail in $\DDD$ contains an induced subgraph isomorphic to  a $t$-sail which does not contain  a $C_m$ cycle. Thus the subclass $\DDD \cap Free(C_m)$ still contains  a $t$-sail for arbitrarily large $t$ and has unbounded tree-width, which  contradicts $\DDD$ being minimal.
\end{proof}

%
%
%
%
%
%
\section{Concluding remarks} \label{Sect:Conclude}

This paper is, as far as we know, the first time an attempt has been made to use combinatorics on words in the study of treewidth. We believe the results are sufficient enough to justify further use of this technique. Likewise, path-star hereditary graph classes seem to be significant in respect of the study of tree-width and clique-width in sparse graph classes and warrant a more thorough study.

We have shown that path-star graph classes defined by nested words block large walls and large line graphs of walls. However, we have not resolved whether there are other such words, or whether, if we forbid a large wall and a large line graph of a wall in a path-star graph class $\RRR^{\alpha}$, then $\alpha$ contains a large nested subword.

Although we have added a new object, a $t$-sail, the identification of a full list of boundary objects that are obstructions to bounded tree-width in hereditary graph classes is still some way from being achieved. Our approach has been to consider graphs of bounded arboricity, in particular, those graphs of arboricity two constructed using forests of paths and stars. We believe this approach could be a fruitful way to identify further boundary objects.
%
%
%
%
%
%

\paragraph{Acknowledgements} 
The author would like to thank Robert Brignall for valuable discussions relating to this work and providing useful feedback which helped to improve the presentation, and is also grateful to the anonymous referees whose careful review of an earlier draft led to several significant improvements.

 ------------------------------------------------------------------------------------------------
\bibliographystyle{plain}
\bibliography{refs}

\def\cprime{$'$}
\begin{thebibliography}{10}

\bibitem{aboulker:tree-width-even-hole-free:}
P.~Aboulker, I.~Adler, E.~J. Kim, N.~L.~D. Sintiari, and N.~Trotignon.
\newblock On the tree-width of even-hole-free graphs.
\newblock {\em European J. Combin.}, 98:Paper No. 103394, 21, 2021.

\bibitem{abrishami:induced_subgraphs_VIII:}
T.~Abrishami, B.~Alecu, M.~Chudnovsky, S.~Hajebi, and S.~Spirkl.
\newblock Induced subgraphs and tree decompositions viii. excluding a forest in
  (theta, prism)-free graphs, 2023.
\newblock arXiv:2301.02138.

\bibitem{abrishami:induced_subgraphs_VII:}
T.~Abrishami, B.~Alecu, M.~Chudnovsky, S.~Hajebi, and S.~Spirkl.
\newblock Induced subgraphs and tree decompositions {VII}. {B}asic obstructions
  in {$H$}-free graphs.
\newblock {\em J. Combin. Theory Ser. B}, 164:443--472, 2024.

\bibitem{abls:minimal-classes-of:}
A.~Atminas, R.~Brignall, V.~Lozin, and J.~Stacho.
\newblock Minimal classes of graphs of unbounded clique-width defined by
  finitely many forbidden induced subgraphs.
\newblock {\em Discrete Applied Mathematics}, 295:57--69, 2021.

\bibitem{Bonamy:sparse_graphs_log_treewidth:}
M.~Bonamy, E.~Bonnet, H.~D\'{e}pr\'{e}s, L.~Esperet, C.~Geniet, C.~Hilaire,
  S.~Thomass\'{e}, and A.~Wesolek.
\newblock Sparse graphs with bounded induced cycle packing number have
  logarithmic treewidth.
\newblock In {\em Proceedings of the 2023 {A}nnual {ACM}-{SIAM} {S}ymposium on
  {D}iscrete {A}lgorithms ({SODA})}, pages 3006--3028. SIAM, Philadelphia, PA,
  2023.

\bibitem{brignall_cocks:uncountable:}
R.~Brignall and D.~Cocks.
\newblock Uncountably many minimal hereditary classes of graphs of unbounded
  clique-width.
\newblock {\em Electron. J. Combin.}, 29(1):Paper No. 1.63, 27, 2022.

\bibitem{brignall:framework_minimal_classes:}
R.~Brignall and D.~Cocks.
\newblock A framework for minimal hereditary classes of graphs of unbounded
  clique-width.
\newblock {\em SIAM J. Discrete Math.}, 37(4):2558--2584, 2023.

\bibitem{collins:infinitely-many:}
A.~Collins, J.~Foniok, N.~Korpelainen, V.~Lozin, and V.~Zamaraev.
\newblock Infinitely many minimal classes of graphs of unbounded clique-width.
\newblock {\em Discrete Appl. Math.}, 248:145--152, 2018.

\bibitem{courcelle:handle-rewritin:}
B.~Courcelle, J.~Engelfriet, and G.~Rozenberg.
\newblock Handle-rewriting hypergraph grammars.
\newblock {\em J. Comput. System Sci.}, 46(2):218--270, 1993.

\bibitem{courcelle:upper-bounds-to:}
B.~Courcelle and S.~Olariu.
\newblock Upper bounds to the clique width of graphs.
\newblock {\em Discrete Appl. Math.}, 101(1-3):77--114, 2000.

\bibitem{davies:oberwolfach2022:}
T.~Davies.
\newblock \uppercase{P}roblem session 4, \uppercase{T}reewidth of hereditary
  classes, p66, \uppercase{O}berwolfach technical report
  \uppercase{DOI}:10.4171/\uppercase{OWR}/2022/1.

\bibitem{dawar:clique_width:}
A.~Dawar and A.~Sankaran.
\newblock M{SO} undecidability for hereditary classes of unbounded clique
  width.
\newblock In {\em 30th {EACSL} {A}nnual {C}onference on {C}omputer {S}cience
  {L}ogic}, volume 216 of {\em LIPIcs. Leibniz Int. Proc. Inform.}, pages Art.
  No. 17, 17. Schloss Dagstuhl. Leibniz-Zent. Inform., Wadern, 2022.

\bibitem{diestel:graph-theory5:}
R.~Diestel.
\newblock {\em Graph theory}, volume 173 of {\em Graduate Texts in
  Mathematics}.
\newblock Springer, Berlin, fifth edition, 2017.

\bibitem{geelen:the-grid-theorem:}
J.~Geelen, Kwon O-J., McCarty R., and Wollan P.
\newblock The grid theorem for vertex-minors.
\newblock {\em Journal of Combinatorial Theory, Series B}, 2020.

\bibitem{gurski-wanke:tree-width-clique-width-Knn:}
F.~Gurski and E.~Wanke.
\newblock The tree-width of clique-width bounded graphs without {$K_{n,n}$}.
\newblock In {\em Graph-theoretic concepts in computer science ({K}onstanz,
  2000)}, volume 1928 of {\em Lecture Notes in Comput. Sci.}, pages 196--205.
  Springer, Berlin, 2000.

\bibitem{hickingbotham:treewidth_circlegraphs:}
R.~Hickingbotham, F.~Illingworth, B.~Mohar, and D.~R. Wood.
\newblock Treewidth, circle graphs, and circular drawings.
\newblock {\em SIAM J. Discrete Math.}, 38(1):965--987, 2024.

\bibitem{korhonen:bounded_degree:}
T.~Korhonen.
\newblock Grid induced minor theorem for graphs of small degree.
\newblock {\em J. Combin. Theory Ser. B}, 160:206--214, 2023.

\bibitem{lozin:minimal-classes:}
V.~V. Lozin.
\newblock Minimal classes of graphs of unbounded clique-width.
\newblock {\em Ann. Comb.}, 15(4):707--722, 2011.

\bibitem{lozin:well-quasi-ordering-does-not:}
V.~V. Lozin, I.~Razgon, and V.~Zamaraev.
\newblock {\em Well-quasi-ordering Does Not Imply Bounded Clique-width}, pages
  351--359.
\newblock Springer Berlin Heidelberg, Berlin, Heidelberg, 2016.

\bibitem{lozin:tw_dichotomy:}
V.V. Lozin and I.~Razgon.
\newblock Tree-width dichotomy.
\newblock {\em European J. Combin.}, 103:Paper No. 103517, 8, 2022.

\bibitem{nash-williams:decompositions:}
C.~St. J.~A. Nash-Williams.
\newblock Decomposition of finite graphs into forests.
\newblock {\em J. London Math. Soc.}, 39:12, 1964.

\bibitem{pohoata:unavoid:}
A.C. Pohoata.
\newblock "\uppercase{U}navoidable induced subgraphs of large graphs"
  \uppercase{S}enior thesis, \uppercase{P}rinceton \uppercase{U}niversity
  (2014).

\bibitem{robertson:graph-minors-ii:}
N.~Robertson and P.~D. Seymour.
\newblock Graph minors. {II}. {A}lgorithmic aspects of tree-width.
\newblock {\em J. Algorithms}, 7(3):309--322, 1986.

\bibitem{robertson:graph-minors-v:}
N.~Robertson and P.~D. Seymour.
\newblock Graph minors. {V}. {E}xcluding a planar graph.
\newblock {\em J. Combin. Theory Ser. B}, 41(1):92--114, 1986.

\bibitem{robertson:graph-minors-xx:}
N.~Robertson and P.~D. Seymour.
\newblock Graph minors. {XX}. {W}agner's conjecture.
\newblock {\em J. Combin. Theory Ser. B}, 92(2):325--357, 2004.

\bibitem{sintiari:ttf-free:}
N.~L.~D. Sintiari and N.~Trotignon.
\newblock ({T}heta, triangle)-free and (even hole, {$K_4$})-free graphs---part
  1: {L}ayered wheels.
\newblock {\em J. Graph Theory}, 97(4):475--509, 2021.

\bibitem{weissauer:blocks:}
D.~Wei{\ss}auer.
\newblock On the block number of graphs.
\newblock {\em SIAM J. Discrete Math.}, 33(1):346--357, 2019.

\bibitem{Zeckendorf:Fibonacci-representation:}
E.~Zeckendorf.
\newblock Repr\'{e}sentation des nombres naturels par une somme de nombres de
  {F}ibonacci ou de nombres de {L}ucas.
\newblock {\em Bull. Soc. Roy. Sci. Li\`ege}, 41:179--182, 1972.

\end{thebibliography}

\end{document}